\documentclass[preprint,12pt]{elsarticle} 
\usepackage{booktabs}
\usepackage{amssymb}
\usepackage[colorlinks]{hyperref}   
\usepackage{xcolor}
\usepackage{amsthm}
\usepackage{mathtools}
\usepackage{enumitem}
\usepackage{graphicx}
\usepackage{subfigure}
\usepackage{autobreak}
\usepackage{natbib}
\usepackage[margin=2cm]{geometry}
\usepackage{multirow}
\usepackage{floatrow}
\usepackage{colortbl}
\usepackage{threeparttable}
\usepackage{subcaption}
\newtheorem{theorem}{Theorem}

\newtheorem{lemma}{Lemma}
\newtheorem{definition}{Definition}

\newtheorem{proposition}{Proposition}
\newtheorem{assumption}{Assumption}
\usepackage{graphicx}
\newfloatcommand{capbtabbox}{table}[][\FBwidth]
\usepackage{multirow}
\usepackage{amsthm,amsmath,amssymb}
\usepackage{mathrsfs}
\usepackage{bm}

\usepackage{amssymb}
\usepackage{algorithm} 
\usepackage{algpseudocode} 
\usepackage{multirow}
\usepackage{makecell}
\usepackage{tabulary}
\usepackage{booktabs}
\usepackage{mathtools}

\setcitestyle{authoryear,open={(},close={)}}
\definecolor{orange}{rgb}{1,0.5,0}
\definecolor{red}{RGB}{198,0,35}
\definecolor{amberseldef}{rgb}{1.0, 0.49, 0.0}
\definecolor{ceruleanblue}{rgb}{0.16, 0.32, 0.75}
\definecolor{amber}{rgb}{1.0, 0.49, 0.0}
\definecolor{dodgerblue}{rgb}{0.12, 0.56, 1.0}
\definecolor{pureblue}{rgb}{0, 0, 1.0}

\definecolor{blue}{rgb}{0.0, 0.28, 0.67}
\newcommand{\blue}[1]{\textcolor{black}{#1}}

\def\hmath$#1${\texorpdfstring{{\rmfamily\textit{#1}}}{#1}}

\makeatletter
\def\ps@pprintTitle{%
   \let\@oddhead\@empty
   \let\@evenhead\@empty
   \let\@oddfoot\@empty
   \let\@evenfoot\@oddfoot
}
\makeatother

\def\scrN{\mathcal{N}}
\def\scrF{\mathcal{F}}
\def\tpo{{t+1}}
\def\tmo{{t-1}}
\def\fullT{{1:T}}

\def\KL{\mathrm{KL}}
\def\sube{_{\epsilon}}
\def\scrE{{\mathbb{E}}}
\def\asto{\xrightarrow{\mathrm{a.s.}}}
\def\etahat{{\hat{\eta}}}
\def\thetahat{{\hat{\theta}}}
\def\rhohat{{\hat{\rho}}}

\def\otild{{\Tilde{o}}}
\def\logit{{\text{logit}}}

\DeclareMathOperator*{\argmin}{arg\,min}

\AtBeginDocument{\hypersetup{citecolor= pureblue,linkcolor = pureblue,urlcolor = dodgerblue}}
\begin{document}


\begin{frontmatter}


\address[1]{Department of Civil and Environmental Engineering, University of Michigan, Ann Arbor, MI, USA}
\address[2]{Department of Industrial and Operations Engineering, University of Michigan, Ann Arbor, MI, USA}
\address[3]{Department of Civil and Environmental Engineering, Duke University, Durham, NC, USA}
\address[4]{Department of Civil Engineering, Stony Brook University, Stony Brook, NY, USA}
\cortext[cor1]{Corresponding author: yafeng@umich.edu }

\author[1,2]{Minghui Wu}
\author[1,2]{Yafeng Yin \texorpdfstring{\corref{cor1}}{}}
\author[3]{Jerome P. Lynch}
\author[4]{Zhichen Liu}

\title{Statistical Inference of Day-to-Day Traffic Dynamics}

\begin{abstract}
Day-to-day traffic dynamics are widely used to model flow evolution due to travelers’ learning and adjustment behavior, yet empirical analysis of these models often relies on descriptive calibration with limited inferential content. This paper develops a statistical inference framework for day-to-day route choice dynamics based on a stochastic individual-level adjustment model. The framework enables uncertainty quantification and formal inference for behavioral parameters from trajectory data. We establish identifiability and consistency under mild conditions, and extend the framework to accommodate demand variation, user heterogeneity through a hierarchical structure, and anonymized observability caused by privacy constraints on trajectory data. Simulation studies demonstrate good finite-sample performance, calibrated uncertainty, and robustness to model misspecification. Empirical analyses of controlled laboratory experiments and real-world trajectory data from Ann Arbor, Michigan, show that the framework can generate novel behavioral insights across settings: it reveals the inadequacy of a purely inter-day learning model once en-route information is introduced, recovers systematic behavioral differences across participant types, and uncovers meaningful day-to-day learning together with substantial demand variation in real-world commuting behavior.
\end{abstract}


\end{frontmatter}

\section{Introduction}
Day-to-day traffic dynamics are widely used to model flow evolution due to travelers’ learning and adjustment behavior. 
From a system-level perspective, day-to-day traffic dynamics provide valuable insights into how traffic patterns evolve over time \citep{watling2013modelling}. Through theoretical analysis of its convergence and stability, researchers aim to establish a behavioral justification for the concept of user equilibrium \citep{wardrop1952road}, a cornerstone for transportation network analysis and planning. At the individual level, day-to-day route choices reflect travelers’ learning and adaptation behaviors. One example is the Smith dynamics \citep{smith1984stability}, which captures user inertia through gradual route swapping, as travelers shift from higher-cost routes to lower-cost routes. More broadly, the literature has incorporated a range of behavioral mechanisms about how travelers learn, adjust, and respond to congestion over time \citep{horowitz1984stability,guo2011bounded,he2012modeling}.

As illustrated in Figure~\ref{fig:literature}, existing day-to-day models can be broadly divided into two types: deterministic and stochastic. Both are primarily concerned with the {forward} problem, namely, specifying how behavioral parameters govern the day-to-day evolution of traffic states. Using path flow as the system state for illustration, deterministic models treat the state on a given day as a deterministic function of the previous state, thereby representing population-average dynamics \citep{watling1999stability}. Some of these models capture emergent behavior from micro-behavioral rules \citep{horowitz1984stability}, whereas others directly
model aggregate adjustment dynamics \citep{smith1984stability}. Stochastic models, by contrast, explicitly characterize the random realization of travelers' daily choices, so that the state on each day is drawn from a probability distribution \citep{davis1993large, cantarella1995dynamic, hazelton2002day, hazelton2004computation, QI2024100123}.

\begin{figure}[!ht]
    \centering
    \includegraphics[width=0.4\linewidth]{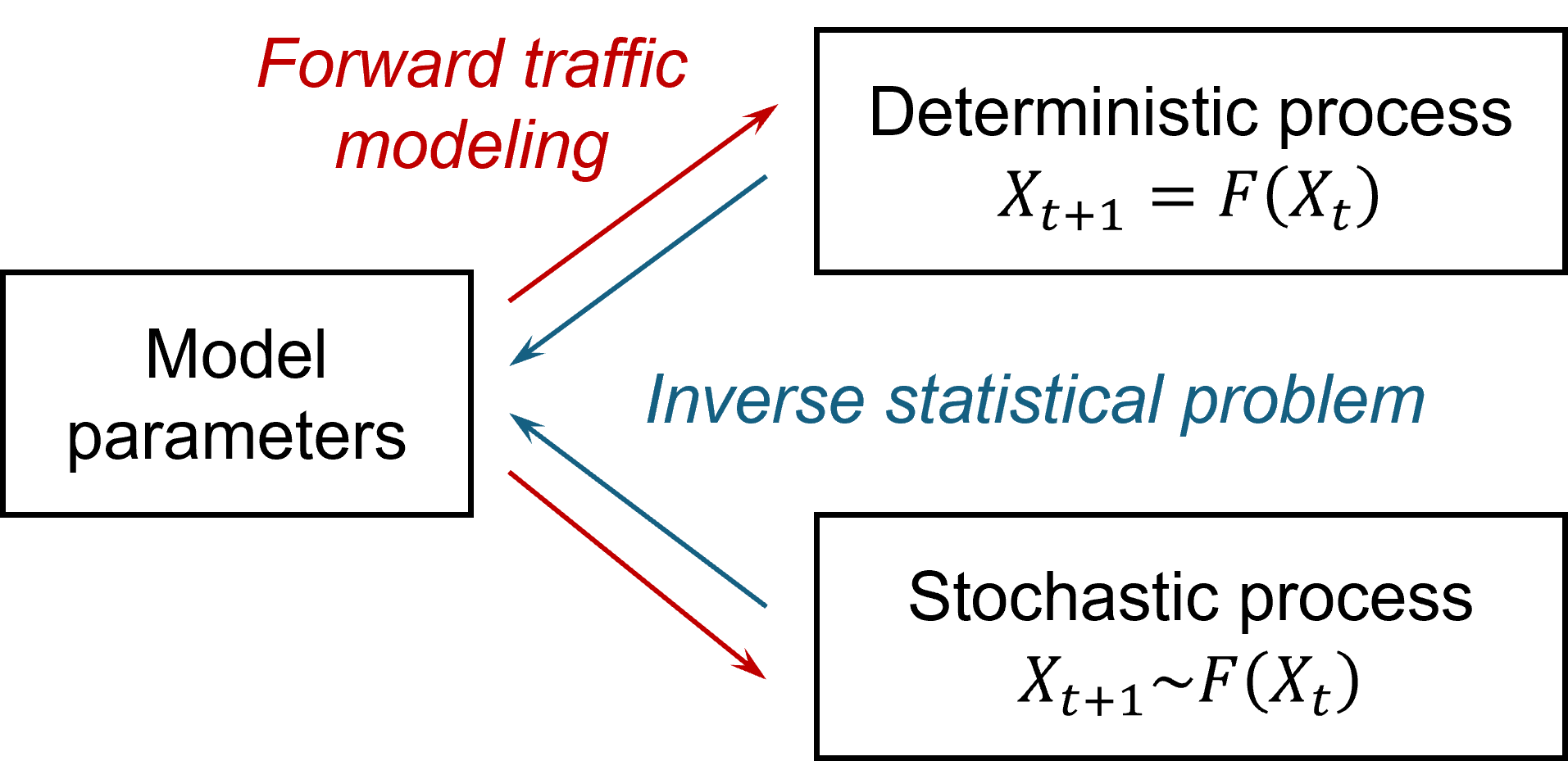}
    \caption{Day-to-day modeling and estimations}
    \label{fig:literature}
\end{figure}

Despite substantial theoretical progress in forward modeling, the inverse problem, recovering behavioral parameters from empirical data, remains a significant and
underexplored challenge. A large body of empirical work has relied on deterministic forward models, reducing estimation to the minimization of a prediction error, \blue{an approach that} essentially functions as \textit{model calibration} \citep{mahmassani1990dynamic, mahmassani2000transferring,
srinivasan2003analyzing, ye2018exploration,qi2023investigating,
guo2011bounded, he2012modeling, cheng2019surrogate}. Although this approach is useful and easy to implement, it has important limitations. \blue{Deterministic models by themselves do not provide the probability distributions or sampling uncertainty required for formal statistical inference. Consequently, differences in calibrated behavioral parameters across populations or conditions cannot, by themselves, be interpreted as genuine differences in underlying adjustment behavior, as they may reflect finite-sample variation. Without an explicit probabilistic framework, there is no principled basis for distinguishing these possibilities. 
}

Stochastic day-to-day models offer a natural remedy by providing the probabilistic foundation needed for \textit{statistical inference} on behavioral parameters. However, only a small number of studies have moved in this direction. Seminal work by \citet{parry2013bayesian, hazelton2016statistical} jointly inferred path flows and behavioral parameters from observed link flows. However, their system state, path flows, have long been considered unobservable and non-identifiable from standard loop detector data \citep{yang2018stochastic}. In such settings, multiple latent path-flow processes, potentially associated with different behavioral parameters, can be observationally equivalent at the link-flow level. This makes the identification of the underlying behavioral parameters difficult to establish. Admittedly, other day-to-day models are formulated using observed link flows as system states \citep{he2010link}. However, this comes at the cost of behavioral interpretability, since individual route choice behavior is no longer modeled explicitly.


Recent advances in connected vehicle and GPS technologies have fundamentally altered this situation. High-resolution trajectory data are now increasingly available, and unlike loop detector counts, they offer direct visibility into individual path choices and even how travelers adjust routes from day to day. This new data regime resolves the long-standing observability bottleneck and provides sufficient structure for the identification of behavioral parameters. At the same time, enhanced observability opens the door to richer behavioral modeling. Travelers naturally exhibit heterogeneous responses to experienced costs: long-term residents with stable travel habits may have low sensitivity to cost differences, whereas newcomers still learning the network may be far more responsive. Trajectory data can allow us to track individuals over time and model this heterogeneity explicitly, moving well beyond the average flow patterns that have constrained much of the prior literature.

To address these opportunities, this paper proposes a statistical inference framework for individual day-to-day route choice behavior. We first introduce an individual-level route adjustment model grounded in established day-to-day dynamics. Building on this model, we develop a Bayesian estimation framework that provides full uncertainty quantification rather than point estimates alone. We theoretically establish model identifiability and prove consistency of the estimation procedure under mild regularity conditions, thereby providing a rigorous foundation for inference.
We further extend the framework to accommodate anonymized observability caused by common data privacy constraints, and to incorporate traveler heterogeneity through a hierarchical structure. For validation, the proposed method is thoroughly examined in a simulated small-scale network to evaluate estimation performance and verify theoretical properties. It is then applied to lab experiment datasets and high-resolution trajectory data from Ann Arbor, Michigan, to generate novel behavioral insights.

Our paper develops a complete inferential framework for the inverse problem of stochastic day-to-day models. Specifically, we make the following contributions:
\begin{itemize}
\item \textit{Rigorous inferential foundation}: We establish identifiability and consistency for a Bayesian framework applied to stochastic day-to-day models, properties that prior calibration-based approaches cannot provide. This clarifies both {what} is being estimated and {whether} reliable estimation is theoretically feasible.
\item \textit{Uncertainty quantification.} By adopting a Bayesian approach, we move beyond point estimation and characterize the full posterior distribution over behavioral parameters, enabling principled statistical inference.
\item \textit{Traveler heterogeneity.} We explicitly model and estimate heterogeneous behavioral parameters across individuals through a hierarchical structure, recovering the population distribution of adjustment behavior rather than a single aggregate value. To the best of our knowledge, this is the first study to do so.
\item \textit{Empirical insights.} We apply the framework to both controlled lab experiments and real-world trajectory data from Ann Arbor, Michigan, generating novel behavioral insights.
\end{itemize}

The remainder of this paper is structured as follows. 
Section~\ref{sec:model} presents the model, our estimation approach, and theoretical results. Section~\ref{sec:extension} extends the model to anonymized observability and incorporates user heterogeneity. Section~\ref{sec:simulation} provides simulation studies on a synthetic network, and Section~\ref{sec:routing} applies the proposed method to analyze behaviors in lab environments and real-world routing scenarios. Finally, Section~\ref{sec:conclusion} concludes the paper. 

\section{Model and Estimation}
To facilitate the presentation of our approach, we begin with a simplified setting in which all commuters are assumed to be homogeneous. This specification is commonly referred to as the pooled model. Section \ref{sec:extension} later extends the framework to a hierarchical model that captures user heterogeneity. As we focus on the commuting problem, we use the terms travelers and commuters interchangeably.

\label{sec:model}
\subsection{Individual Day-to-Day Choice Model}
\label{ssec:individual}

Let $[N]=\left\{1,...,N\right\}$ denote the study group of commuters, a sample from the total demand, where $N$ is the total number of observed commuters. To better present our approach, we first consider a single origin-destination (OD) pair. Let $[M]=\left\{1,...,M\right\}$ denote the set of available paths and $[T]=\left\{1,...,T\right\}$ the time horizon. The travel cost of path $i$ on day $t$ is denoted by $c_t(i)$ and is treated as exogenously given. The daily cost vector is written as $c_t=\left\{c_t(i)\right\}_{i\in[M]}$, and we assume that the cost is always bounded by a finite constant $C$. In practice, demand may also vary over time. To accommodate this feature, we treat non-travel as a virtual alternative, denoted as path 0. 

We use a random variable $X_t^n$ to represent the path chosen by commuter $n$ on day $t$, including the virtual non-travel path, with support $\left\{0,1,...,M\right\}$. To simplify notations, we denote the full choice sequence of commuter $n$ over the time horizon by $X^n_\fullT=[X_1^n,...,X_T^n]$, which is a random vector on $\left\{0,1,...,M\right\}^T$. The collection of choice sequences of all commuters is denoted by $X_{1:T}=[X^1_\fullT, ..., X^N_\fullT]$, defined on $\left\{0,1,...,M\right\}^{T\times N}$.

On each day, each commuter chooses not to travel with a constant probability $\rho\in(0,1)$, which is driven by exogenous needs such as remote work, rather than cost-based elasticity. The resulting choice probabilities are
\begin{equation}
\label{eq:individual_0}
    P(X_t^n=0|\rho)=\rho,
\end{equation}

For physical paths, initially, each commuter is assumed to have limited information about path costs and other travelers’ behavior. We therefore set the initial perceived cost for each physical path to $V_1^n(i) = 0$ for all commuters $n$ and paths $i\in [M]$. Alternatively, the initial valuation could be set to free-flow travel time without affecting the structure of the model. After observing realized travel costs, commuters update their perceived costs using an exponential smoothing rule:
\begin{equation}
\label{eq:individual_1}
    V_{t+1}^n(i)= (1-\eta) V_t^n(i) + \eta c_t(i),  \forall i\in[M], n\in[N],
\end{equation}
where $\eta\in(0,1)$ is the learning rate that captures the sensitivity of each commuter to newly observed travel costs.

Given the perceived costs, each commuter chooses a path according to a multinomial logit model, conditioned on traveling on that day. Specifically, for $i\in[M]$,
\begin{equation}
\label{eq:individual_2}
   P(X_t^n=i | \eta, \theta, \rho, \scrF_{t-1})=(1-\rho) \frac{e^{-\theta V_t^n(i)}}{\sum_{m\in[M]} e^{-\theta V_t^n(m)}},
\end{equation}
where $\theta>0$ is the scale parameter, and $\sigma$-algebra $\scrF_\tmo$ refers to the flow of information over time up to day $t-1$, which essentially consists of the history of the travel costs. 

For notational convenience, we denote the probabilities in Equations~(\ref{eq:individual_0}) and~(\ref{eq:individual_2}) by $p_t^n(i),i\in\left\{0,...,M\right\}$. In the pooled model, all commuters are assumed to share common behavioral parameters and therefore identical choice probabilities, denoted by $p_t(i)$ without the commuter index.

Following the daily choice process, the probability of observing a full choice sequence for a single commuter, $x^n_\fullT\in \left\{0,1,...,M\right\}^T$, is given by
\begin{equation}
\begin{aligned}
    &P(X^n_{1:T}=x_{1:T}^n|\eta, \theta,\rho,c_\fullT)=  \prod_{t=1}^T p_t^n(x_t^n) .
\end{aligned}
\end{equation}

We further assume that commuters make their daily route choices independently, conditional on their perceived costs. Although commuters influence one another through congestion effects embedded in the realized travel costs, conditional independence holds at the decision stage given the perceived valuations. Under this assumption, the probability of observing a full set of choices across all commuters, $x_\fullT\in\left\{0,1,...,M\right\}^{T\times N}$, is
\begin{equation}
\label{eq:singleOD}
\begin{aligned}
    &P(X_\fullT=x_\fullT|\eta, \theta,\rho, c_\fullT) =\prod_{n=1}^N \prod_{t=1}^T p_t^n(x_t^n).
\end{aligned}
\end{equation}

This framework extends naturally to multiple commodities under the assumption that commuters associated with different OD pairs make independent decisions based on their own perceived costs. With a slight abuse of notations, let $N^w$ denote the number of commuters within OD pair $w\in W$, and let $X_{1:T}^{w,n}$ denote the choice sequence of commuter $n$ within OD pair $w$. Then, the joint probability becomes
\begin{equation}
\label{eq:multipleOD}
\begin{aligned}
    &P(X_\fullT=x_\fullT|\eta, \theta,\rho,c_\fullT) \\
    = &\prod_{w\in W} \prod_{n=1}^{N^w} P(X^{w,n}_{1:T}=x_{1:T}^{w,n}|\eta, \theta,\rho,c_\fullT).
\end{aligned}
\end{equation}
Adding OD pairs increases computation complexity in the log-likelihood only linearly. From a computational perspective, multiple OD pairs are equivalent to having more commuters within a single OD pair. As a result, the presence of multiple commodities does not fundamentally alter the structure of the model. For notational simplicity, we therefore focus on the single-OD setting in the remainder of this section.


The proposed model describes route adjustment through two components, an exponential smoothing rule for updating perceived costs and a logit model for translating perceived costs into choice probabilities. 
This structure is grounded in the behavioral foundation of Horowitz dynamics \citep{horowitz1984stability}:
\begin{equation}
    p_\tpo^{w,k}=\eta  c_t^{w,k} + (1-\eta) p_t^{w,k}, \ \forall w\in W, k\in P^w, t>0 \label{eq:horowitz1}
\end{equation}
\begin{equation}
    f^{w,k}_\tpo=D^w \frac{e^{-\theta p_\tpo^{w,k}}}{\sum_{k'\in P^w }e^{-\theta p_\tpo^{w,k'}}} \label{eq:horowitz2}.
\end{equation} 
Here, $p_t^{w,k}$ and $c^{w,k}_t$ refer to the perceived and true cost on path $k\in P^w$ of OD pair $w\in W$ on day $t$, respectively. $f^{w,k}_t$ represents the path flow and $D^w$ refers to the demand. 

Despite sharing similar behavioral components, Horowitz dynamics, as well as most existing day-to-day models, are formulated at the aggregate path-flow level. In addition to allowing for varying demand, a key distinction of our proposed model is complete observability at the individual level: the daily path choices of each studied commuter are directly observed, which is available in high-resolution trajectory data and is also naturally satisfied in laboratory route-choice experiments. By retaining individual choice trajectories rather than only aggregate counts, the proposed formulation provides the foundation needed to infer behavioral parameters and, in Section~\ref{sec:extension}, to accommodate user heterogeneity.

Nonetheless, Proposition~\ref{pp:horowitz} shows that this richer individual-level model remains consistent with the classical aggregate perspective as the number of commuters increases. Similar observations have also been made by \citet{davis1993large} and \citet{watling1999stability}.
Proofs throughout the paper are provided in Appendix \ref{app:proof}.

\begin{proposition}
\label{pp:horowitz}
When all commuters follow Equations~(\ref{eq:individual_0})--(\ref{eq:individual_2}) and $\rho=0$, the expected path flows follow Horowitz dynamics~(\ref{eq:horowitz1})--(\ref{eq:horowitz2}). Moreover, as the number of commuters tends to infinity, the empirical proportion of commuters choosing each path converges to the corresponding proportion implied by Horowitz dynamics.
\end{proposition}

The structure of the model above is summarized in Figure \ref{fig:pool}.
\begin{figure}[!ht]
    \centering
    \includegraphics[width=0.4\linewidth]{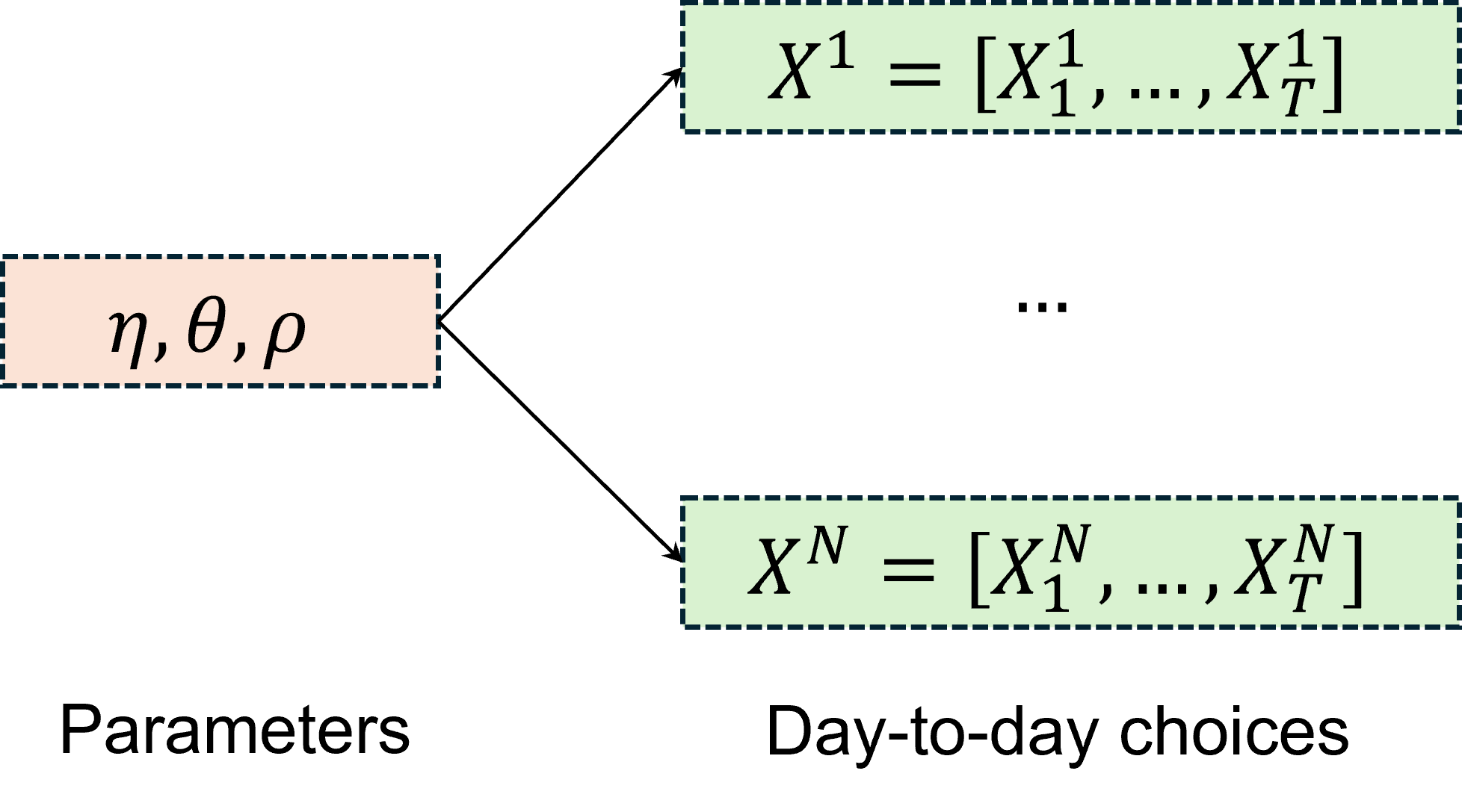}
    \caption{Illustration of the pooled model}
    \label{fig:pool}
\end{figure}

\subsection{Estimation algorithm}\label{ssec:estimation}

We now turn to estimation of the parameters $\eta,\theta,\rho$ after observing a realized set of choice trajectories, $x_\fullT\in \left\{0,1,...,M\right\}^{T\times N}$. 

At first glance, the individual-level formulation appears to induce a state space that grows exponentially with the number of commuters. However, this combinatorial explosion does not enter likelihood evaluation. For estimation, we only need to evaluate the probability of the realized individual trajectories observed in the data, rather than enumerate all possible population states. Because the joint likelihood factorizes across commuters and days under conditional independence, computation scales with the observed data rather than with the size of the full state space.

Specifically, to ensure numerical stability and computational efficiency, we work with the log-probabilities:
\begin{equation}
    \log p_t^n(0)=\log \rho,
\end{equation}
\begin{equation}
    \log p_t^n(i)=\blue{\log(1-\rho)}-\theta V_t^n(i)-\log \sum_{m=1}^M e^{-\theta V_t^n(m)}
\end{equation}
for all $n\in[N], t\in[T],i\in[M]$. The complete-data log-likelihood for the realization is therefore
\begin{equation}
\label{eq:loglikelihood}
\begin{aligned}
    \log L(\eta,\theta,&\rho)  \\
    = \sum_{n=1}^N\sum_{t=1}^T \Bigg(& I_t^n(0) \log\rho+
    \sum_{m=1}^M I_t^n(m) \log p_t^n(m) \Bigg),
\end{aligned}
\end{equation}
where $I_t^n(m)=\textbf{1}\left\{x_t^n=m\right\},m\in\{0,1,...,M\}$ are indicator variables. The resulting likelihood function is differentiable with respect to $(\eta,\theta,\rho)$. 

We adopt a Bayesian inference approach, as it is more robust in settings with limited or noisy data. We assume independent prior distributions for the parameters, so that $p(\eta,\theta,\rho)=p(\eta)p(\theta)p(\rho)$. The
choice of prior families is motivated by the parameter domains. Since $\eta, \rho \in
(0,1)$, we assign logit-normal priors, which place mass on the unit interval while remaining flexible in shape. Since $\theta > 0$, we assign a log-normal prior, which places mass on the positive real lin. Thus, the posterior distribution satisfies
\begin{equation}
    p(\eta,\theta,\rho| x_\fullT, c_\fullT) \propto L(\eta,\theta,\rho) p(\eta)p(\theta)p(\rho).
\end{equation}
In practice, we compute the log-posterior:
\begin{equation}
\label{eq:log-posterior}
\begin{aligned}
    \log p(\eta,&\theta,\rho| x_\fullT, c_\fullT) =\log L(\eta,\theta,\rho) \\
    &+\log p(\eta)+\log p(\theta)+\log p(\rho) +\text{const}.
\end{aligned}
\end{equation}
The Bayes estimator is taken as the posterior mean. For example,
\begin{equation}
    \etahat=\int \eta p(\eta,\theta,\rho|x_\fullT,c_\fullT) d\eta d\theta d\rho,
\end{equation} 
and similarly for $\thetahat$ and $\rhohat$. 

Beyond point estimation, statistical inference requires uncertainty quantification. We report marginal Highest Density Intervals (HDIs) as credible intervals (CIs) for each parameter, which is the narrowest range that can cover certain proportions of the posterior. For a confidence level $\alpha$ (e.g., 95\%), the HDI for $\eta$ is defined as
\begin{equation}
    HDI_\alpha (\eta)=\left\{\eta:p(\eta|x_\fullT,c_\fullT)\geq e_\alpha(\eta) \right\},
\end{equation}
where $p(\eta|x_\fullT,c_\fullT)$ refers to the marginal posterior, and the threshold $e_\alpha(\eta)$ satisfies:
\begin{equation}
    \int_{\left\{\eta: p(\eta|x_\fullT,c_\fullT) \geq e_\alpha(\eta)\right\} }p(\eta|x_\fullT,c_\fullT) d\eta =\alpha.
\end{equation}
Similar for the other two parameters. 

Because the likelihood and posterior do not admit closed-form expressions and must be evaluated recursively through the latent valuation updates, direct posterior evaluation is infeasible. We therefore rely on sampling-based methods that approximate the posterior via empirical distributions. Given posterior draws $\left\{(\eta^{(s)}, \theta^{(s)},\rho^{(s)}\right\}_{s=1}^S$, then
\begin{equation}
    \etahat=\frac{1}{S} \sum_{s=1}^S \eta^{(s)}.
\end{equation}
The $\alpha$-credible interval is
\begin{equation}
    [\eta^{(k)}, \eta^{(k+\lfloor{\alpha S\rfloor})}],
\end{equation}
where 
\begin{equation}
    k = \argmin_{1\leq s\leq S-\lfloor{\alpha S\rfloor}} \left(\eta^{(s+\lfloor{\alpha S\rfloor})}-\eta^{(s)}\right).
\end{equation}
The same procedure applies to the other two parameters.

Because $\theta$ and $V$ interact multiplicatively in the likelihood, the posterior over $(\eta, \theta)$ can exhibit strong correlations: a higher $\theta$ and a lower perceived cost difference can produce similar choice probabilities, creating elongated, curved posterior ridges that standard random-walk Markov chain Monte Carlo (MCMC) traverses inefficiently. We therefore adopt the No-U-Turn Sampler (NUTS) \citep{hoffman2014no}, an adaptive variant
of Hamiltonian Monte Carlo (HMC) \citep{girolami2011riemann}. HMC uses gradient
information to propose distant moves along the posterior geometry, greatly reducing
the autocorrelation of samples compared to random-walk methods, especially in correlated
posteriors. Readers are referred to \citet{hoffman2014no} for details of the algorithm's implementation.


\subsection{Theoretical results}
The estimation algorithm above is well-defined only if the parameters it targets are identifiable and the estimator is consistent. We now establish both properties formally.
\subsubsection{Identifiability}
Identifiability characterizes what parameters can be uniquely learned from the data, independently of the estimation method or sample size.
\begin{definition}[Identifiability]
\label{def:identifiability}
For a given cost sequence $c_\fullT$, the model is identifiable if 
\begin{equation}
\begin{aligned}
   & P(X_\fullT=x_\fullT|\eta, \theta, \rho, c_\fullT) \\
    = & P(X_\fullT=x_\fullT|\eta', \theta',\rho', c_\fullT)
\end{aligned}
\end{equation}
for all $x_\fullT\in \left\{0,1,...,M\right\}^{T\times N}$ 
indicates $\eta=\eta',\theta=\theta',\rho=\rho'$. 
\end{definition}

If identifiability fails, multiple distinct parameter values generate the same likelihood functions. Hence, no amount of data or estimation methods can tell them apart. In such cases, estimation becomes fundamentally ill-posed. We establish identifiability under a rather weak condition on the cost sequence.
\begin{assumption}[Dynamic richness]
\label{ass:rich}
\blue{Suppose $T\geq 3$. There exist two distinct routes $i,j\in[M]$ and a day $t\in\{1,\ldots,T-2\}$ such that $ \Delta c_t(i,j)\neq \Delta V_1(i,j) $}
, where $\Delta c_t(i,j)=c_t(i)-c_t(j)$ and $\Delta V_1(i,j)=V_1(i)-V_1(j)$.    
\end{assumption}

\begin{theorem}
\label{thm:basic}
\blue{Under Assumption~\ref{ass:rich}, the model with parameter $\eta,\theta,\rho$ is identifiable.}
\end{theorem}



\subsubsection{Consistency}

Identifiability is a property of the model and data-generating process: it ensures that the parameter-to-likelihood mapping is injective. Consistency is a property of the estimator: it guarantees convergence to the true parameter values as more data becomes available.

\begin{definition}[Consistency]
Let $\phi\in \Theta$ denote the model parameters with true value $\phi_0$. 
\blue{The estimator is consistent if 
$\Vert \hat{\phi}_T-\phi_0\Vert \asto 0$ as $T\to\infty$, where $\hat{\phi}_T$ is the posterior-mean estimator based on an observation horizon of length $T$.}

\end{definition}

Identifiability alone does not guarantee consistency. For example, if the cost sequence satisfies Assumption \ref{ass:rich} only over a finite initial period and becomes constant thereafter, learning effectively stops, which prevents the estimator from converging to the true values. This motivates a stronger requirement: the cost sequence must remain persistently informative throughout the horizon.

\begin{assumption}[Persistent excitations]
\label{ass:persist}
Let $\phi_0=(\eta_0, \theta_0,\rho_0)\in \Theta $ denote the true parameters. The cost sequence $c_\fullT$ satisfies that for any $\epsilon>0$, there exists $\kappa(\epsilon)>0$ such that for all $\phi\in\Theta$ with $\Vert \phi-\phi_0\Vert\geq \epsilon$:
\begin{equation}
    \blue{\liminf_{T\to \infty} \inf_{\Vert \phi-\phi_0\Vert\geq\epsilon}\frac{1}{T} \sum_{t=1}^T \KL(p_t(\cdot|\phi_0)\Vert p_t(\cdot|\phi)) \geq \kappa(\epsilon) \quad a.s.,}
\end{equation}
where $p_t(\cdot |\phi)$ refers to the choice probability distribution under parameter $\phi$, and $\KL(\cdot|\cdot)$ \blue{represents the Kullback--Leibler(KL) divergence}.
\end{assumption}

Intuitively, this assumption ensures that any parameter $\phi$ sufficiently far from the truth keeps generating choice distributions that are persistently distinguishable from those under $\phi_0$, that is, the KL divergence between them does not vanish. \blue{This persistent separation enables the posterior distribution to eventually reject incorrect parameter values.}

Under this condition and standard regularity assumptions, consistency can be established. 

\begin{theorem}
\label{thm:consistent}
Suppose:
\begin{enumerate}
\item \blue{The parameter space $\Theta= [\eta_{min},\eta_{max}]\times[\theta_{min},\theta_{max}]\times [\rho_{min}, \rho_{max}] \subset (0,1)\times(0,\infty)\times (0,1)$};
\item True parameter $\phi_0$ is in the interior of $\Theta$;
\item The prior is continuous and satisfies $p(\phi_0)>0$.
\end{enumerate}
Then, under Assumption \ref{ass:persist}, {the estimator is consistent}.
\end{theorem}

\subsection{Endogenizing initial values}
The baseline model above assumes exogenously specified initial valuations, such as zeros or the free-flow travel times. While reasonable for newcomers with no prior experience, this assumption may be unrealistic for long-term residents who have already accumulated information.

A natural remedy is to treat the initial valuations $V_1(i)$ as additional model parameters to be estimated. However, this creates a non-identifiability: because multinomial logit choice probabilities depend only on relative costs, any global shift in the initial valuations leaves the likelihood unchanged.

\begin{proposition}
\label{pp:endogV-nonidentify}
For any feasbile $(\eta,\theta,\rho)$, the parameter sets $(\eta,\theta, \rho, V_1)$ and $(\eta,\theta,\rho,V_1')$, where $V_1'(i)=V_1(i)+K$ for any constant $K\neq 0$ and $i\in[M]$, are observationally equivalent and thus not identifiable.
\end{proposition}

To resolve this issue, we parameterize initial conditions using relative differences $\delta(j)=\Delta V_1(j,1)$, $j=2,...,M$. This pins down the location of the initial valuations up to the irrelevant global shift, introducing $M - 1$ additional parameters. Identifying these additional degrees of freedom requires a somewhat stronger
condition on the cost sequence.

\begin{assumption}[Stronger dynamic richness]
\label{ass:richer}
\blue{Suppose $T\geq 3$. The cost difference sequence $\Delta c_t(\cdot,1) = \big( \Delta c_t(2,1),\ldots,\Delta c_t(M,1) \big)^\top$ satisfies the following conditions:
\begin{enumerate} 
\item There exists a day $s\in\{1,\ldots,T-2\}$ such that $\Delta c_s(\cdot,1)\neq 0.$
\item There is no scalar $r\in\mathbb{R}$ such that $\Delta c_{t+1}(\cdot,1) = r\Delta c_t(\cdot,1), \ t=1,\ldots,T-2.$
\end{enumerate}}
\end{assumption}

\blue{This assumption ensures the cost sequence provides sufficient temporal variation to distinguish the initial offsets $\delta$ from the learning rate $\eta$. The condition remains relatively mild in practice, as cost sequences with moderate temporal variation will generally satisfy it. The following theorem establishes identifiability for the model with endogenized initial valuations.}
\begin{theorem}
\label{thm:endogenize}
Under Assumption \ref{ass:richer}, the model with parameters $\eta,\theta,\rho,\delta(2),...,\delta(M)$ is identifiable.
\end{theorem}

\section{Model Extensions}
\label{sec:extension}
So far, we have established a statistical inference framework for the pooled model under full observability of homogeneous travelers. In this section, we extend the framework to incorporate two more realistic features: user heterogeneity and anonymized observability.

\subsection{User Heterogeneity}
The pooled model characterizes all commuters using a single set of parameters, assuming homogeneous behavior. While this yields tractable estimation, it cannot recover the {distribution} of behavioral parameters across individuals, even though this distribution is itself a primary object of interest.
We therefore introduce a hierarchical Bayesian model that treats individual parameters as random draws from a population distribution, and estimates that distribution directly from the data.

Before proceeding, it is worth clarifying why the pooled posterior distribution cannot simply be reinterpreted as the population distribution of behavioral parameters. Although this may seem natural, it is generally misleading. To illustrate, consider a coin-flipping experiment in which $p$ denotes the probability of heads. Suppose the data are generated either (i) by a homogeneous population with $p=1/2$, or (ii) by a heterogeneous population in which each individual has $p\in\left\{0,1\right\}$ with equal probability. Both scenarios produce identical $\text{Bernoulli}(1/2)$ observations. As the sample size grows, the posterior distribution for $p$ concentrates at $1/2$ in both cases, even though the underlying population distributions of $p$ are fundamentally different. The same issue arises here: a pooled estimate of $\eta$ tells us the average learning rate across commuters, but reveals nothing about
whether some commuters are fast learners and others are slow. Recovering that heterogeneity requires a model that explicitly represents it.


\subsubsection{Model}
As demonstrated in Figure \ref{fig:hier}, we assume each traveler $n$ draws an individual parameter vector $\phi^n=(\eta^n, \theta^n,\rho^n)$ from a population distribution governed by hyperparameters $H$. Given $\phi^n$, traveler $n$'s day-to-day choices follow the same individual model as in Section~\ref{ssec:individual}. We denote the hyperpriors by $p(H)$.

\begin{figure}[!ht]
    \centering
    \includegraphics[width=0.55\linewidth]{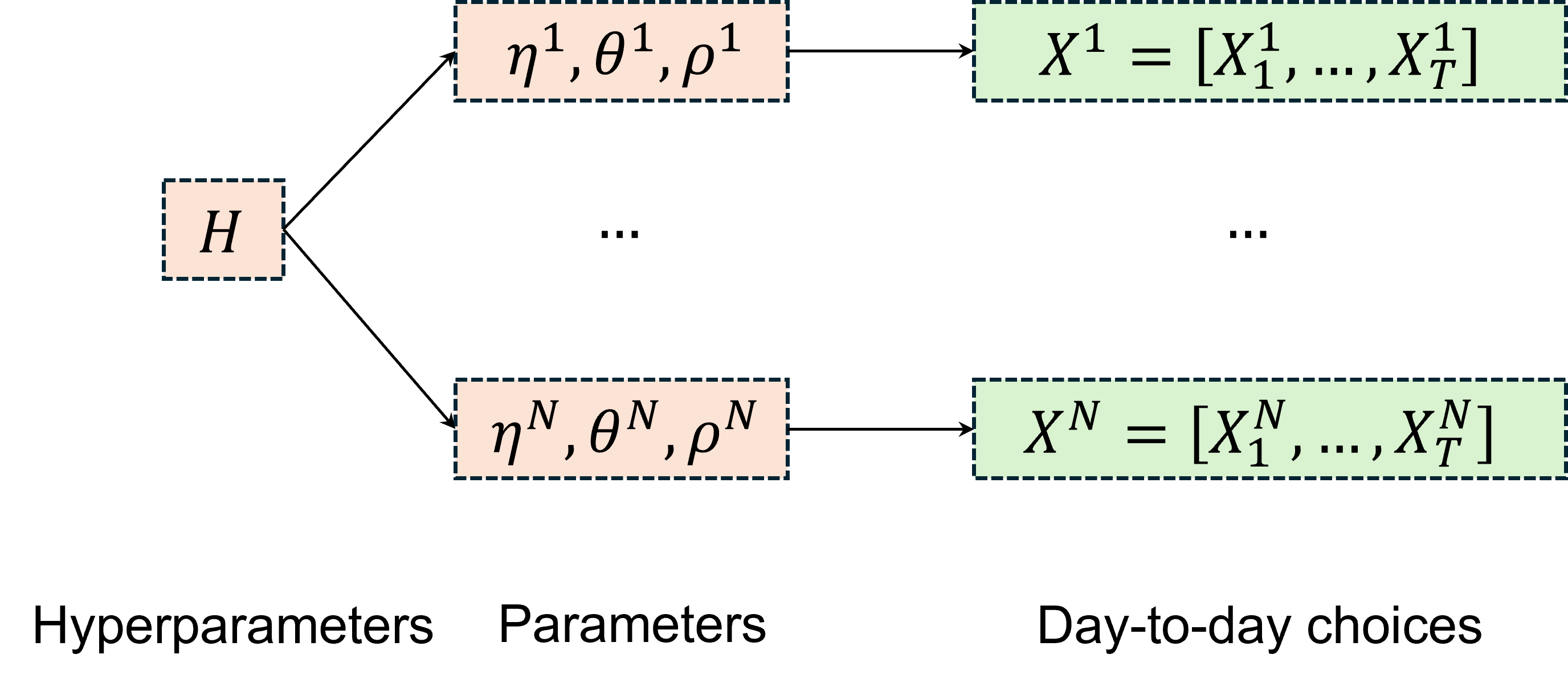}
    \caption{Illustration of the hierarchical model}
    \label{fig:hier}
\end{figure}

The choice of population distributions is flexible. A convenient specification that respects parameter constraints while allowing asymmetric shapes is:
\begin{itemize}
\item Logit-normal for $\eta^n$ and $\rho^n$
\begin{equation}
    \log\left( {\eta^n}/{(1-\eta^n)} \right) \sim \scrN(\mu_\eta, \sigma_\eta^2),
\end{equation}
\begin{equation}
    \log\left( {\rho^n}/{(1-\rho^n)} \right) \sim \scrN(\mu_\rho, \sigma_\rho^2);
\end{equation}
\item Log-normal for $\theta^n$
\begin{equation}
    \log(\theta^n) \sim \scrN(\mu_\theta,\sigma_\theta^2).
\end{equation}
\end{itemize}

Under this specification, the hyperparameters are $H=(\mu_\eta,\sigma_\eta,\mu_\theta,\sigma_\theta,\mu_\rho,\sigma_\rho)$. We place normal hyperpriors on the location parameters $\mu_\eta,\mu_\theta,\mu_\rho$ and half-normal hyperpriors on the scale parameters $\sigma_\eta,\sigma_\theta,\sigma_\rho$. \blue{The half-normal distribution has support on the nonnegative real line and provides regularization for the population-level scale parameters.} 

With heterogeneous parameters, the day-$t$ choice probabilities become
\begin{equation}
    P(X_t^n=0|\rho^n,\eta^n,\theta^n, \scrF_\tmo)=\rho^n,
\end{equation}
\begin{equation}
    P(X_t^n=i|\rho^n,\eta^n,\theta^n, \scrF_\tmo)=  \frac{(1-\rho^n)e^{-\theta^n V_t^n(i)}}{\sum_{j\in[M]} e^{-\theta^n V_t^n(j)}}
\end{equation}
for all $i\in[M], t\in[T]$. We retain the notation $p_t^n(i)$ for these probabilities with minor abuse of notation.

Formal identification of the hyperparameters $H$ requires a mixture identifiability argument that combines Theorem~\ref{thm:basic} with properties of the logit-normal family; we leave a rigorous treatment for future work. In practice, $H$ is estimable when both $N$ and $T$ are sufficiently large, as individual parameters become well-identified with large $T$ based on previous theorems, and their population distribution becomes recoverable with large $N$, which will be further illustrated in simulation studies in Section \ref{sec:simulation}.

\subsubsection{Estimation}
The complete-data log-likelihood under the hierarchical model is
\begin{equation}
\label{eq:loglikelihood_hier}
\begin{aligned}
&\log L(\phi^1,\ldots,\phi^N) \\
=&
\sum_{n=1}^N\sum_{t=1}^T
\left[
I_t^n(0)\log\rho^n
+
\sum_{m=1}^M I_t^n(m)\log p_t^n(m)
\right],
\end{aligned}
\end{equation}


The joint posterior over individual parameters and hyperparameters is
\begin{equation}
\begin{aligned}
   & p(\phi^1,...,\phi^N, H| x_\fullT, c_\fullT) \\
     \propto &L(\phi^1,...,\phi^N) p(\phi^1,...,\phi^N|H) p(H),
\end{aligned}
\end{equation}

In practice, we sample jointly from this full posterior using NUTS. Marginal \blue{posterior distributions} for the hyperparameters $H$ and for each individual parameter vector $\phi^n$ are then read off directly from the posterior draws. The Bayes estimator $\hat{H}$ and individual estimates $\hat{\phi}^n$ are taken as posterior means, and credible intervals are computed via HDIs as in Section~\ref{ssec:estimation}.

However, directly sampling over $(H,\phi^1,...,\phi^N)$ can exhibit numerical pathologies, known as Neal’s funnel geometry \citep{neal2003slice}. The problem arises because the geometry of the posterior changes drastically with $H$: when $\sigma_\eta$ is large, the individual parameters $\eta^n$ can vary widely; when $\sigma_\eta$ is small, all $\eta^n$ must lie very close to $\mu_\eta$, creating a narrow ridge of high posterior mass that is difficult for HMC-based samplers to explore accurately. A step size that works well in the wide regime
is far too large for the narrow regime, leading to either poor exploration or rejected proposals.

To improve posterior geometry, we employ a non-centered parameterization for all hierarchical parameters. For example,
\begin{equation}
\label{eq:mixing2}
    \log\left( \frac{\eta^n}{1-\eta^n} \right) = \mu_\eta+\sigma_\eta z_\eta^n, \quad z_\eta^n\sim\scrN(0,1),
\end{equation}
and similar for the $\theta^n$ and $\rho^n$. This formulation preserves the same generative process as the centered model but decouples the latent individual offsets $z^n$ from the scale parameters $\sigma$. The posterior over $z^n$ is now approximately standard normal regardless of $\sigma_\eta$, yielding a well-conditioned geometry that NUTS can explore efficiently across the full range of hyperparameter values.

\subsection{Anonymized Observations}

The preceding sections assumed complete individual-level observability, motivated by the emergence of high-resolution trajectory data. However, privacy regulations may require that user identifiers be reshuffled or anonymized on a daily basis, so that daily route counts are available but individual trajectories cannot be linked across days. Note that this observability issue considered here should not be confused with another separate problem of recovering path flows from link-level observations.

\subsubsection{Observation model}
We introduce the following observation model shown in Figure \ref{fig:partial}. As discussed earlier, we still present the model for a single OD pair, which can be easily extended to multiple commodities.

On each day $t$, we observe an $(M+1)$-dimensional vector $O_t=[O_t(0),...,O_t(M)]$, where each element denotes the number of commuters choosing that route, including the virtual ``not traveling'' option:
\begin{equation}
\label{eq:observation}
    O_t(i)=\sum_{n=1}^N \textbf{1}\left\{X_t^n = i\right\}, i=0,...,M.
\end{equation}
We write $O_t=\left\{O_t(i)\right\}_{i=0}^M$ and denote the full observation sequence as $O=\left\{O_t\right\}_{t=1}^T$. 

\begin{figure}[!ht]
    \centering
    \includegraphics[width=0.8\linewidth]{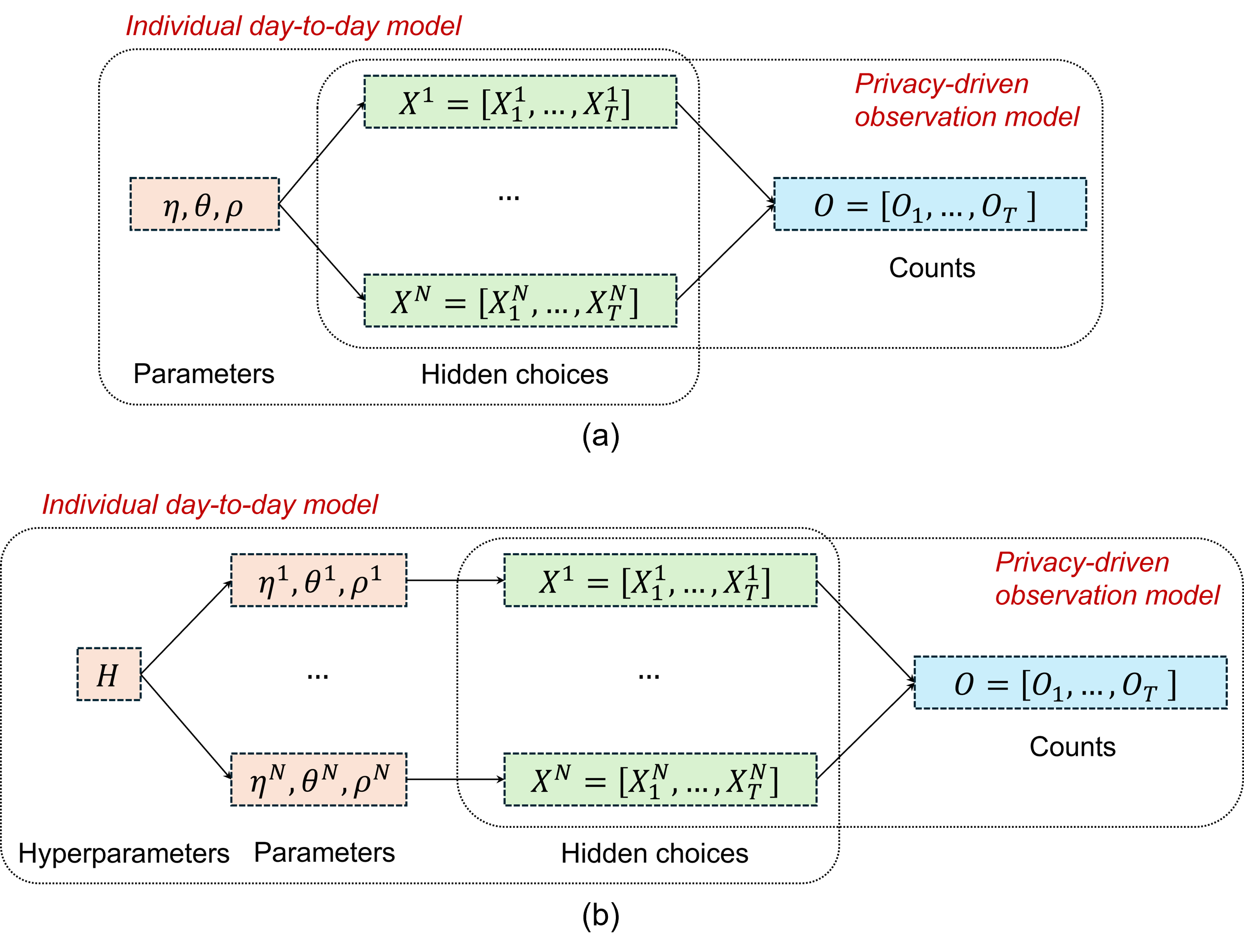}
    \caption{Anonymized observability: (a) pooled model; (b) hierarchical model.}
    \label{fig:partial}
\end{figure}

\subsubsection{Pooled model}
\label{sssec:partial+pool}
Under the pooled model, shown in Figure \ref{fig:partial}(a), all travelers share the same behavioral parameters and initial values. Hence $p_t^n(i)=p_t^{n'}(i)=p_t(i)$ for all $t\in[T], n,n'\in[N]$ and all $i=0,...,M$.

Because all commuters draw independently from the same categorical distribution $p_t=[p_t(0),...,p_t(M)]$, the observed count vector follows a multinomial distribution:
\begin{equation}
\begin{aligned}
&P(O_t=o_t|\eta, \theta,\rho, \scrF_\tmo) \\
=& \frac{N!}{o_t(0)!\cdots o_t(M)!} \prod_{i=0}^M p_t(i)^{o_t(i)} 
\end{aligned}
\end{equation}

Therefore, for a realized observation sequence, the likelihood function is:
\begin{equation}
\begin{aligned}
    &L(\eta,\theta,\rho)= \prod_{t=1}^T P(O_t=o_t|\eta, \theta,\rho, \scrF_\tmo).
\end{aligned}
\end{equation}
Taking logs yields
\begin{equation}
\label{eq:loglikelihood_hier_partial}
\begin{aligned}
    \log L(\eta,&\theta,\rho)= T\log N! \\
    +& \sum_{t=1}^T\sum_{i=0}^M \left( o_t(i)\log p_t(i) - \log o_t(i)! \right).
\end{aligned}
\end{equation}
The log-posterior is then obtained by adding the log-priors as in the fully observable case.

Interestingly, anonymized observability does not weaken identifiability in the pooled model.
\begin{proposition}
\label{pp:partial-pool-identify}
Under Assumption \ref{ass:rich}, the pooled model with parameters $(\eta,\theta,\rho)$ and observation model (\ref{eq:observation}) is identifiable.
\end{proposition}

Similar for the model with endogenized initial values:
\begin{proposition}
\label{pp:partial-pool-endogV-identify}
Under Assumption \ref{ass:richer}, the pooled model with parameters $(\eta,\theta,\rho,\delta(2),...,\delta(M))$ and observation model (\ref{eq:observation}) is identifiable.
\end{proposition}

The key reason is permutation invariance: since all commuters share the same choice distribution under the pooled model, the likelihood is unchanged if we permute the labels of any two commuters. This means that only aggregate counts matter for inference and individual identities carry no additional information. The following theorem formalizes this equivalence and reveals what identifiability in the pooled model truly requires.

\begin{theorem}
\label{thm:equal}
For any realized choice sequence $x_\fullT$ and its corresponding aggregate observations $o_\fullT$, the two posteriors are equal:
\begin{equation}
    p(\eta,\theta,\rho|o_\fullT, c_\fullT) = p(\eta,\theta,\rho|x_\fullT, c_\fullT)
\end{equation}
for all $\eta,\theta,\rho.$
\end{theorem}

This theorem indicates that individual-level data and aggregate count data are informationally equivalent for the purpose of recovering $(\eta, \theta, \rho)$. Therefore, identifiability in day-to-day models hinges not on the granularity of individual tracking, but on the observability of the path flow over time.

\subsubsection{Hierarchical model}
The pooled case is convenient because all commuters share the same choice distribution, yielding multinomial observations. In the hierarchical model with anonymized observability (Figure \ref{fig:partial}(b)), this convenience no longer holds: different travelers may have different parameters and initial values, leading to different choice distributions $p_t^n$. The aggregate count vector $o_t$ now follows a Poisson-multinomial distribution (PMD) \citep{lin2022poisson}:
\begin{equation}
    o_t\sim PMD(p_t^1,...,p_t^N), \ \forall t\in[T].
\end{equation}
The first two moments are
\begin{equation}
    \scrE[o_t] = \sum_{n=1}^N p_t^n, \label{eq:E}
\end{equation}
\begin{equation}
    \text{Cov}(o_t) = \sum_{n=1}^N \left( \text{diag}(p_t^n) - p_t^n (p_t^n)^T \right), \label{eq:Cov}
\end{equation}
where $\text{diag}(p_t^n)$ denotes the diagonal matrix with entries $p_t^n(0),...,p_t^n(M)$. 

However, the PMD does not admit a simple closed-form likelihood, and exact computation is too expensive to embed inside an MCMC sampling loop. We therefore adopt an approximation. During sampling, whenever the likelihood is needed, we approximate $o_t$ by a multinomial random vector
\begin{equation}
    \otild_t\sim \text{Multinomial}(N, \Bar{p}_t),
\end{equation}
where $\Bar{p}_t = (1/N) \sum_{n=1}^N p_t^n$.

\begin{proposition}
\label{pp:approximate}
    For all $t\in[T]$, $\scrE[o_t] = \scrE[\otild_t], \quad \text{Cov}(o_t) \preceq \text{Cov}(\otild_t).$
\end{proposition}

The proposition indicates that this approximation matches the mean exactly but overestimates variance. To see the magnitude, consider the binary case $i\in\{0,1\}$:
\begin{equation}
    \text{Var}(\otild_t(0)) - \text{Var}(o_t(0))
    = N \text{Var}(p_t^n(0)),
\end{equation}
which vanishes as heterogeneity in individual choice probabilities decreases. The direction of the approximation error matters for inference: because the multinomial approximation overestimates variance, the resulting likelihood is more diffuse than the true PMD likelihood, which tends to produce wider and more conservative credible intervals.

\section{Simulation Studies}
\label{sec:simulation}
In this section, we evaluate the proposed estimation approach in a controlled synthetic environment. The goals are to (i) demonstrate the computation, and (ii) empirically validate the theoretical properties established earlier.

\subsection{Experiment Setup}
We generate background traffic using the standard Nguyen–Dupuis (ND) network setting \citep{nguyen1984efficient}, with four OD pairs $1\to2,1\to3,4\to2,4\to3$. Background traffic evolves according to Horowitz dynamics with zero initial valuations. To avoid convergence to a fixed equilibrium and to mimic real-world variability, we add Gaussian noise to the path valuations each day. We simulate the background dynamics for 20 days as a warm start, and we assume observations begin on day 21.

As shown in Figure \ref{fig:ND}, we then introduce a study group of commuters traveling from Node 5 to Node 11, with three feasible paths:  $5\to7\to11,$ $5\to9\to11$, and $5\to6\to10\to11$. Because this group is small relative to background traffic, we assume it does not affect the path costs generated by the background dynamics (i.e., costs are treated as exogenous for the study group).

\begin{figure}[!ht]
    \centering
    \includegraphics[width=0.45\linewidth]{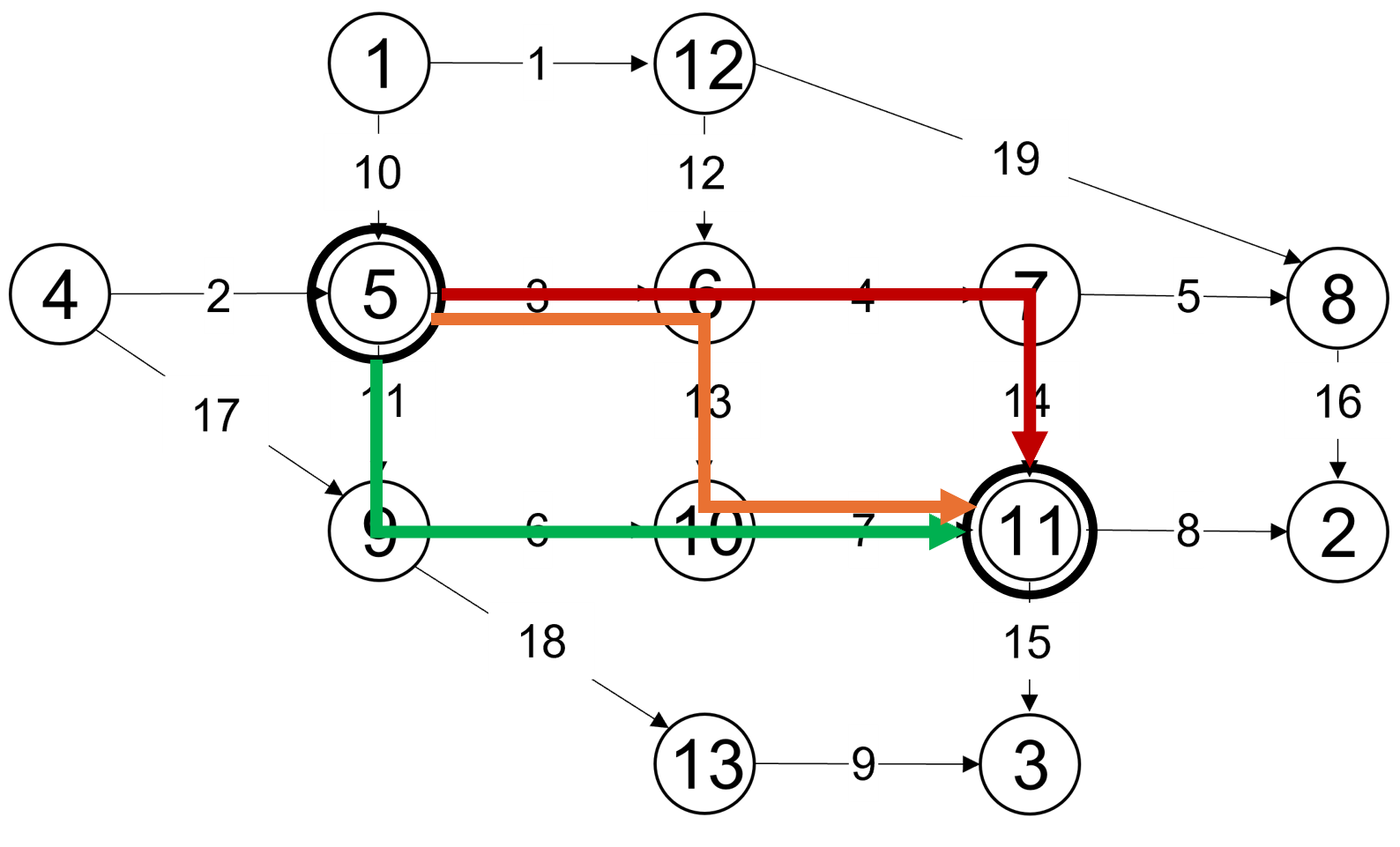}
    \caption{Simulation environment}
    \label{fig:ND}
\end{figure}

We evaluate performance using three metrics:
\begin{itemize}
\item Mean bias: measures the accuracy of the point estimation; 
\item Empirical coverage of 95\% CI: evaluates the calibration of uncertainty. An estimator may be nearly unbiased yet still produce misleading uncertainty estimates, which undermines statistical inference and downstream decision-making.
\item Width of the 95\% CI: among estimators with similar coverage, a narrower interval indicates greater information extraction and higher precision from the same data.
\end{itemize}

\subsection{Pooled Models}
\subsubsection{Estimation results}
We first conduct a parameter recovery study under the pooled model. Specifically, we draw $1{,}000$ ``true'' parameter triples $(\eta^{(s)}, \theta^{(s)}, \rho^{(s)})$ independently from the estimation prior: $\log \left(\frac{\eta^{(s)}}{1-\eta^{(s)}}\right) \sim N(0, 1.5), \log \theta^{(s)}\sim N(0, 1.0), \log \left(\frac{\rho^{(s)}}{1-\rho^{(s)}}\right) \sim N(-2.0, 1.0)$, $s=1,...,1000$.

For each parameter triple, we generate choice sequences under the fixed travel cost sequence while varying (i) the horizon length $T$ and (ii) the number of travelers $N$. For each simulated dataset, we compute posterior point estimates and corresponding 95\% intervals.

Figure \ref{fig:posterior} shows an example of the estimated joint posterior and marginals from a single run. The posterior samples appear approximately Gaussian, suggesting well-behaved inference in this setting. In this example, $\theta$ exhibits a larger deviation than $\eta$ and $\rho$, which is unsurprising because  $\theta$ has broader support and enters multiplicatively with valuations, making it harder to pin down.

\begin{figure}[!ht]
    \centering
    \includegraphics[width=0.8\linewidth]{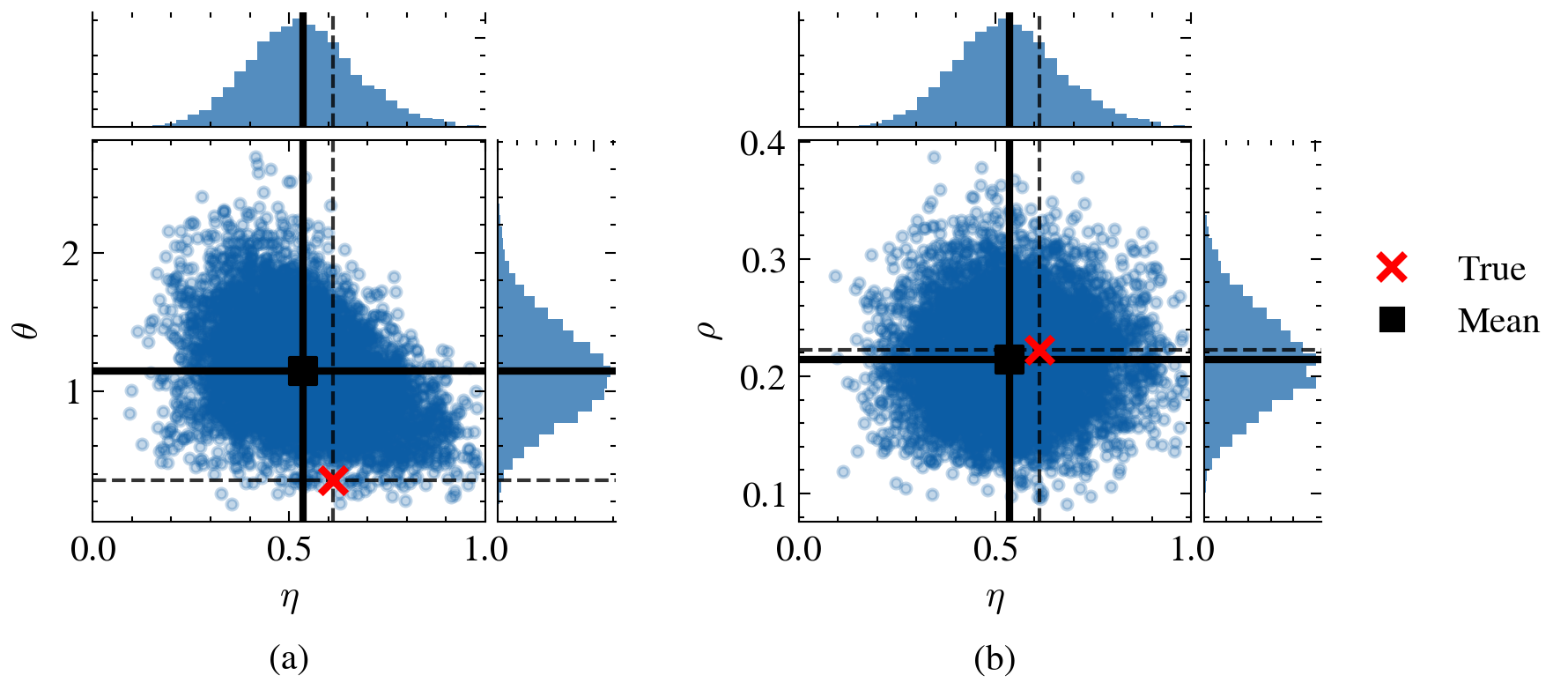}
    \caption{An example for estimated posteriors. (a): $\eta$ vs $\theta$; (b): $\eta$ vs $\rho$.}
    \label{fig:posterior}
\end{figure}

Because any single run is noisy, we summarize performance across all 1,000 trials. We report sampling diagnostics and model-fit checks in Appendix \ref{app:ss:diagnostics}. Figure \ref{fig:bias_plot} reports average bias under different settings. \blue{For the left three panels}, we fix $T=30$ and increase $N$ from 1 to 20; \blue{for the right three panels}, we fix $N=3$ and increase $T$ from 10 to 50. Overall, bias decreases as information increases, consistent with Theorem \ref{thm:consistent}. Some non-monotonicity remains due to finite-sample randomness. 

Interestingly, increasing $N$ is more effective than increasing $T$. Intuitively, longer horizons provide more observations but also propagate latent-valuation uncertainty through the recursive updates, whereas increasing $N$ provides more independent sequences at each $t$. Finally, $\rho$ tends to be easier to estimate because it affects travel or not directly and is less entangled with the other two parameters. 

\begin{figure}[!ht]
    \centering
\includegraphics[width=1\linewidth]{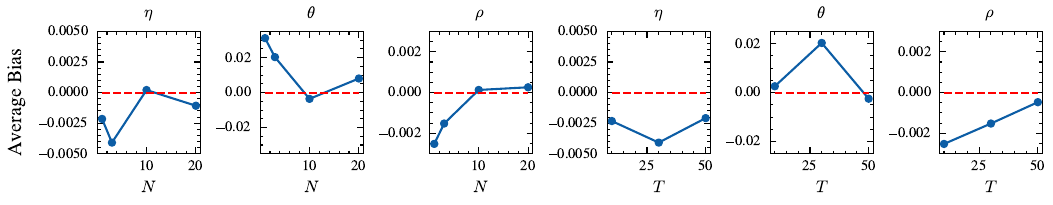}
    \caption{Average bias. Left three panels:  fix $T=30$ and vary $N$. Right three panels: fix $N=3$ and vary $T$. }
    \label{fig:bias_plot}
\end{figure}

Figure \ref{fig:coverage_plot} presents empirical coverage of 95\% CI across simulation settings. Coverage remains close to the nominal level across horizon lengths and numbers of travelers, indicating good finite-sample calibration and reliable uncertainty quantification in these settings.

\begin{figure}[!ht]
    \centering
    \includegraphics[width=0.5\linewidth]{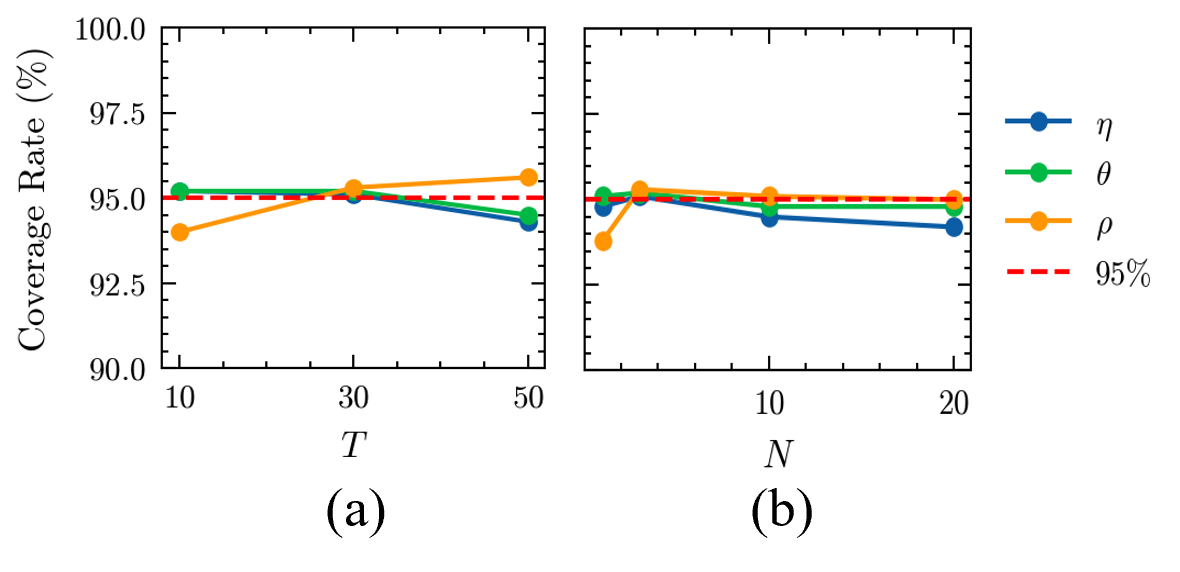}
    \caption{Coverage rate: (a) Fix $N=3$ and vary $T$; (b) Fix $T=30$ and vary $N$. }
    \label{fig:coverage_plot}
\end{figure}

Figure \ref{fig:width_plot} reports the average width of the 95\% CI. Interval widths decrease monotonically as either $T$ or $N$ increases, reflecting increasing posterior concentration and again aligning with the consistency result in Theorem \ref{thm:consistent}. This also indicates that as information accumulates, the estimator becomes more confident.

\begin{figure}
    \centering
    \includegraphics[width=0.9\linewidth]{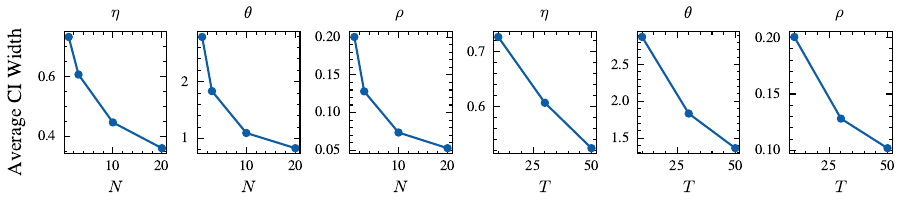}
    \caption{Average CI width. Left three panels: Fix $T=30$ and vary $N$.
    Right three panels: fix $N=3$ and vary $T$. }
    \label{fig:width_plot}
\end{figure}

\subsubsection{Robustness to misspecification}
\blue{We next examine the robustness of the pooled estimator to behavioral misspecification. Specifically, we consider two data-generating processes: the Horowitz model, representing correct specification, and the Smith model \citep{smith1984stability}, representing an alternative day-to-day adjustment mechanism. All simulated datasets are estimated using the proposed Horowitz-based model. For each data-generating process, we consider different levels of behavioral responsiveness, population sizes, and training-horizon lengths. To isolate systematic predictive error from finite-agent sampling variation, we evaluate the root mean squared error (RMSE) over a held-out horizon of $T_{test}=20$ days by comparing predicted choice probabilities with the corresponding true choice probabilities rather than with realized discrete choices. Detailed experiment settings are presented in Appendix \ref{app:ss:robust}.}

\blue{As shown in Figure \ref{fig:held-out}, under correct specification, the held-out RMSE is consistently small and decreases as more data become available. As expected, model misspecification under the Smith data-generating process produces larger prediction errors. Nevertheless, predictive accuracy improves substantially with a longer observation horizon. This pattern is also evident in the extrapolated choice probabilities shown in Figure \ref{fig:smith_extrapolate}: richer training data improves predictive performance and yields reasonably accurate extrapolation even under behavioral misspecification.}

\begin{figure}[!ht]
    \centering
    \includegraphics[width=0.8\linewidth]{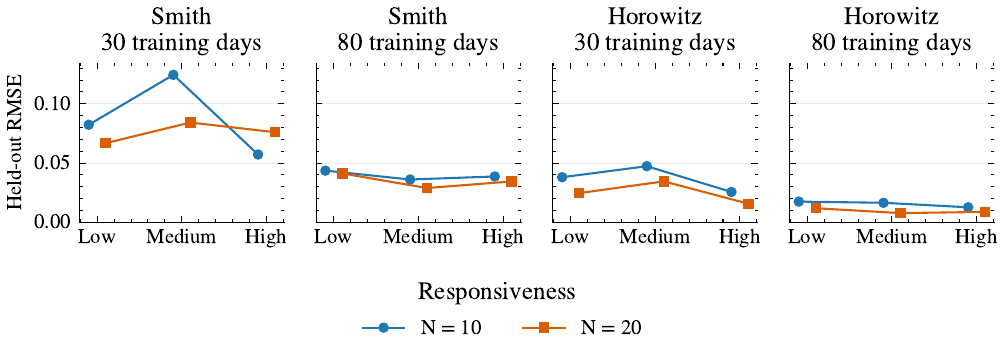}
    \caption{\blue{Held-out choice probabilities RMSE under different simulation settings.}}
    \label{fig:held-out}
\end{figure}

\begin{figure}[!ht]
    \centering
    \includegraphics[width=0.85\linewidth]{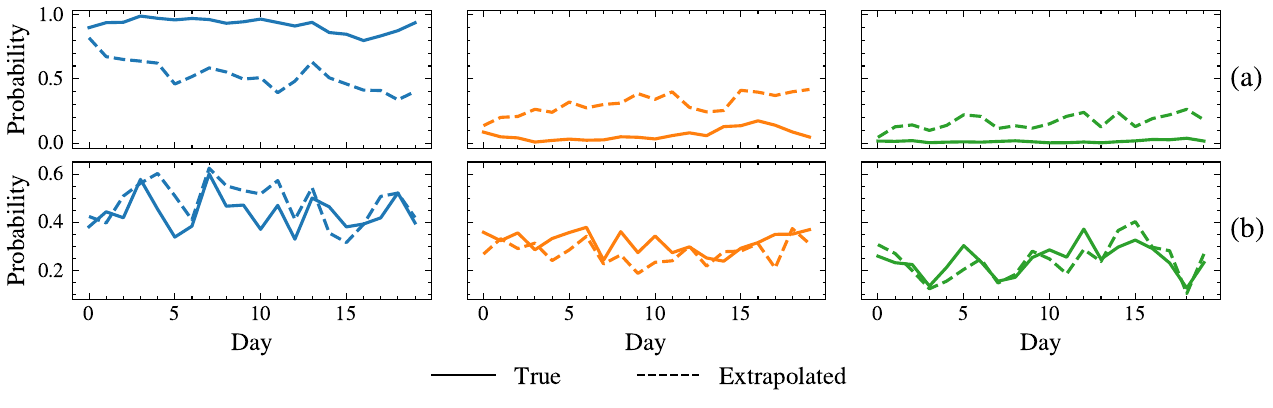}
    \caption{\blue{Extrapolated choice probabilities over $T_{test}=20$ days. (a): $N=10, T_{train}=30$; (b): $N=20, T_{train}=80$.}}
    \label{fig:smith_extrapolate}
\end{figure}

\blue{These results demonstrate that different day-to-day adjustment mechanisms can generate similar aggregate flow patterns under suitable parameter values. Therefore, a good fit under the proposed model does not, by itself, establish that travelers actually learn according to the assumed behavior. Interpretations of the estimated behavioral parameters remain conditional on the model specification. }

\blue{In addition to behavioral misspecification, we also consider (i) prior misspecification, where data are generated from parameter distributions different from the estimation priors, and (ii) behavioral heterogeneity, where data are generated from heterogeneous agents. Across these settings, misspecification increases bias and reduces coverage relative to the correctly specified case, but overall performance degradation remains moderate. Detailed results and discussions are provided in Appendix \ref{app:ss:robust}.}

\subsection{Hierarchical Models}
This section moves beyond the pooled model to directly recover population heterogeneity.

\subsubsection{Estimation results}
\label{sss:hier_results}
We conduct a hyperparameter recovery study. We draw $1{,}000$ true hyperparameter vectors from the distributions in Table \ref{tb:hyperparameter}. Condition on hyperparameters $H^{(s)}$, each traveler $n$ draws their individual parameters via $\log \left(\frac{\eta^{(s,n)}}{1-\eta^{(s,n)}}\right) \sim \scrN(\mu_\eta^{(s)},\sigma_\eta^{(s)}), \log \theta^{(s,n)}\sim \scrN(\mu_\theta^{(s)},\sigma_\theta^{(s)}), \log \left(\frac{\rho^{(s,n)}}{1-\rho^{(s,n)}}\right) \sim \scrN(\mu_\rho^{(s)},\sigma_\rho^{(s)})$, and then generates choices under the fixed cost sequence.
\begin{table}[!ht]
\centering
\begin{tabular}{llll}
\hline
\multicolumn{2}{l}{$\scrN(\mu,\sigma)$} & \multicolumn{2}{l}{$Half\scrN(\sigma)$} \\ \hline
$\mu_\eta^{(s)} $          & $(-1.5, 0.5)$        & $\sigma_\eta^{(s)}$             & $(0.5)$          \\
$\mu_\theta^{(s)}$         & $(0, 0.5)$           & $\sigma_\theta^{(s)}$           &$(0.5)$          \\
$\mu_\rho^{(s)}$           & $(-2,1)$            & $\sigma_\rho^{(s)}$             & $(1)$            \\ \hline
\end{tabular}
\caption{Generating distributions of hyperparameters}
\label{tb:hyperparameter}
\end{table}

Figures \ref{fig:hier_bias_N} and \ref{fig:hier_bias_T} report average bias for the six hyperparameters under varying $N$ and $T$. In general, increasing $N$ systematically improves estimation, which is expected because a larger population provides more information about between-user dispersion. In contrast, increasing $T$ has a much smaller effect. Intuitively, longer horizons help identify each individual’s behavior more precisely, but provide limited additional information about user heterogeneity once individual parameters are already reasonably identified.

\begin{figure}[!ht]
    \centering
    \includegraphics[width=1\linewidth]{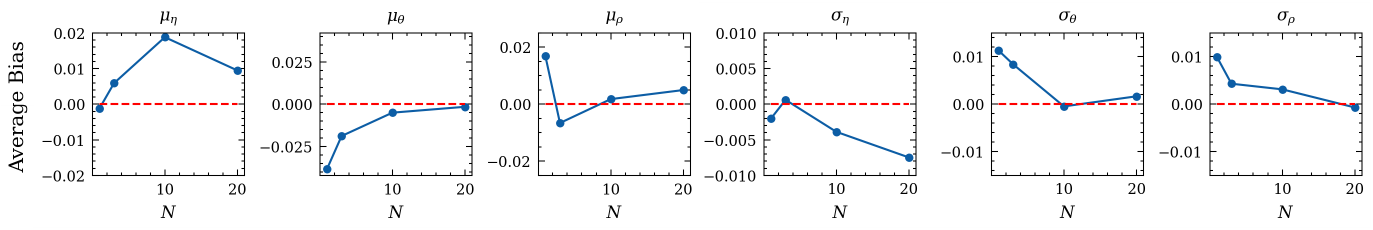}
    \caption{Average bias of hyperparameters against varying $N$.}
    \label{fig:hier_bias_N}
\end{figure}

\begin{figure}[!ht]
    \centering
    \includegraphics[width=1\linewidth]{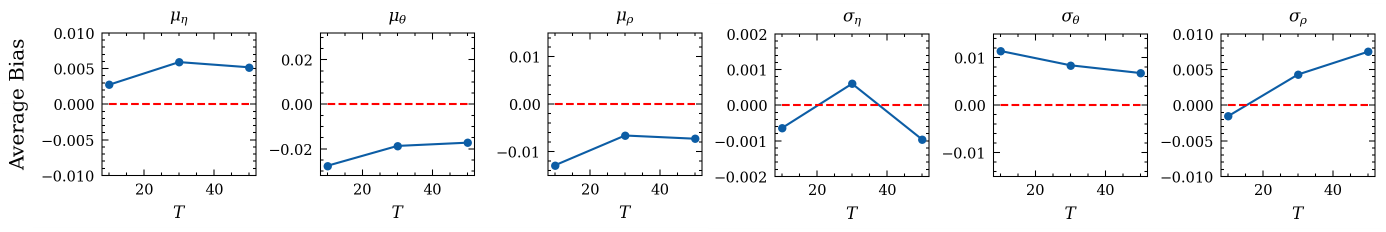}
    \caption{Average bias of hyperparameters against varying $T$.}
    \label{fig:hier_bias_T}
\end{figure}

We further examine uncertainty quantification for hierarchical parameters. Empirical coverage of 95\% credible intervals remains close to the nominal level across most settings, and interval widths decrease monotonically with increasing $N$, reinforcing that population size is the dominant source of information for learning heterogeneity. In contrast, increasing $T$ yields diminishing returns once individual-level behavior is sufficiently identified. Detailed coverage and interval-width results are reported in Appendix \ref{app:ss:hier}.

\subsubsection{Model misspecification}
We next ask whether the hierarchical estimator can recover population-level behavior when the assumed population distributions are misspecified. This matters in practice because the true distributions are rarely known.

We generate individual parameters using: $\eta^{n}\sim Beta(2, 5), \theta^{n}\sim Gamma(2,1), \rho^{n}\sim Beta(2,8)$. 
\blue{Again, we estimate the model using the first $T_{train}$ days and evaluate extrapolation accuracy for the next $T_{test}=20$ days.}

Figure \ref{fig:hier_diff_DGP_extrapolate} shows that predicting aggregate choice probabilities is relatively easy: even with $N=10$ and $T_{train}=30$, the model captures the main fluctuation patterns with minor level shifts. However, recovering the population distribution is substantially harder: with small $N$, the estimated distribution can deviate notably from the truth (Figure \ref{fig:hier_diff_DGP_posterior}(a)). As $N$ increases, the estimator observes a wider range of behaviors, and population distribution estimates improve. Increasing $T_{train}$ further improves both distribution recovery and predictive accuracy. Remarkably, for $N=50$ and $T_{train}=80$, extrapolated choice probabilities nearly overlap with the truth.

\begin{figure}[!ht]
    \centering
    \includegraphics[width=0.8\linewidth]{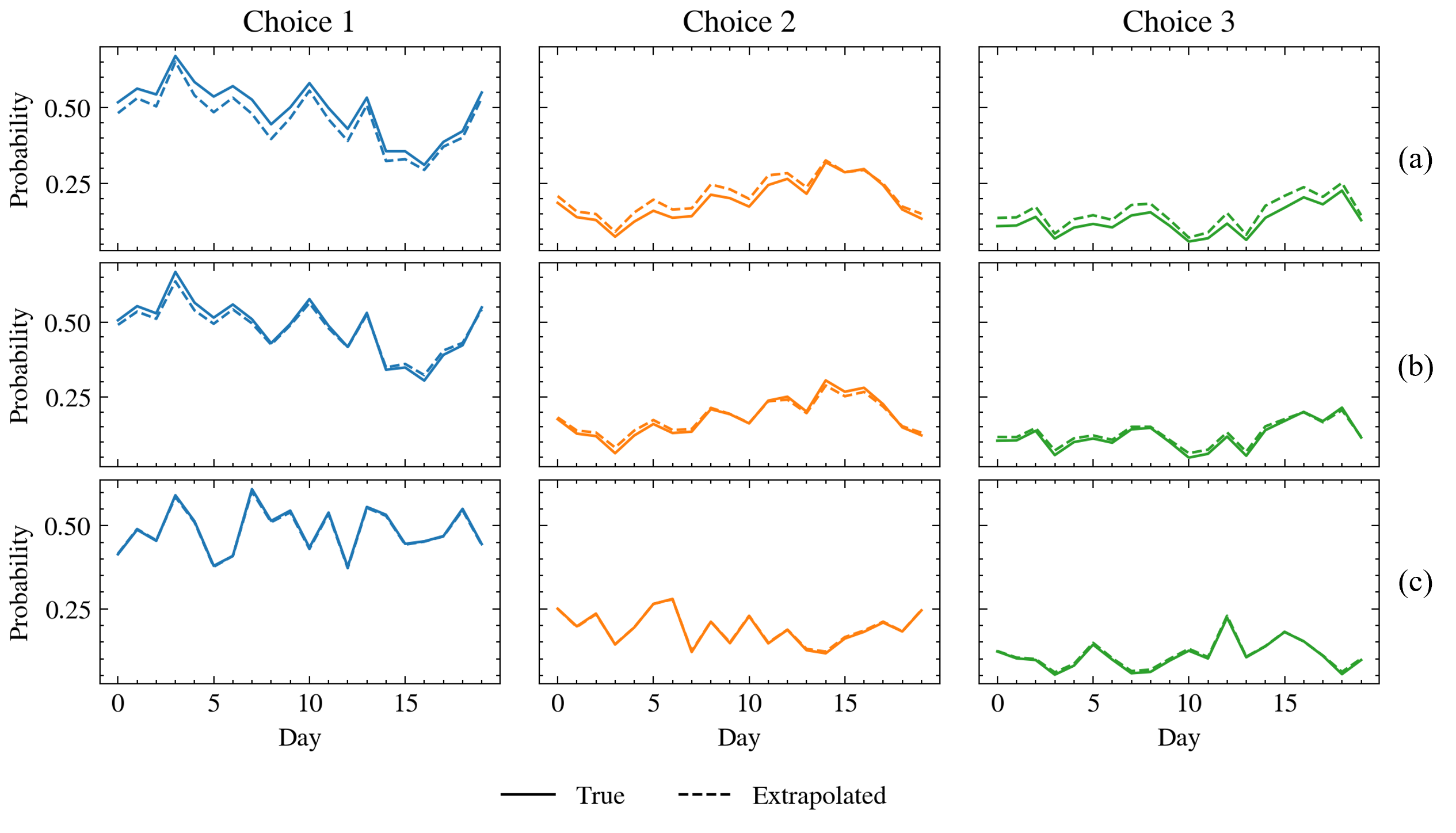}
    \caption{Extrapolated choice probabilities. (a): $N=10, T_{train}=30$; (b): $N=50, T_{train}=30$; (c): $N=50, T_{train}=80$.}
    \label{fig:hier_diff_DGP_extrapolate}
\end{figure}

\begin{figure}[!ht]
    \centering
    \includegraphics[width=0.8\linewidth]{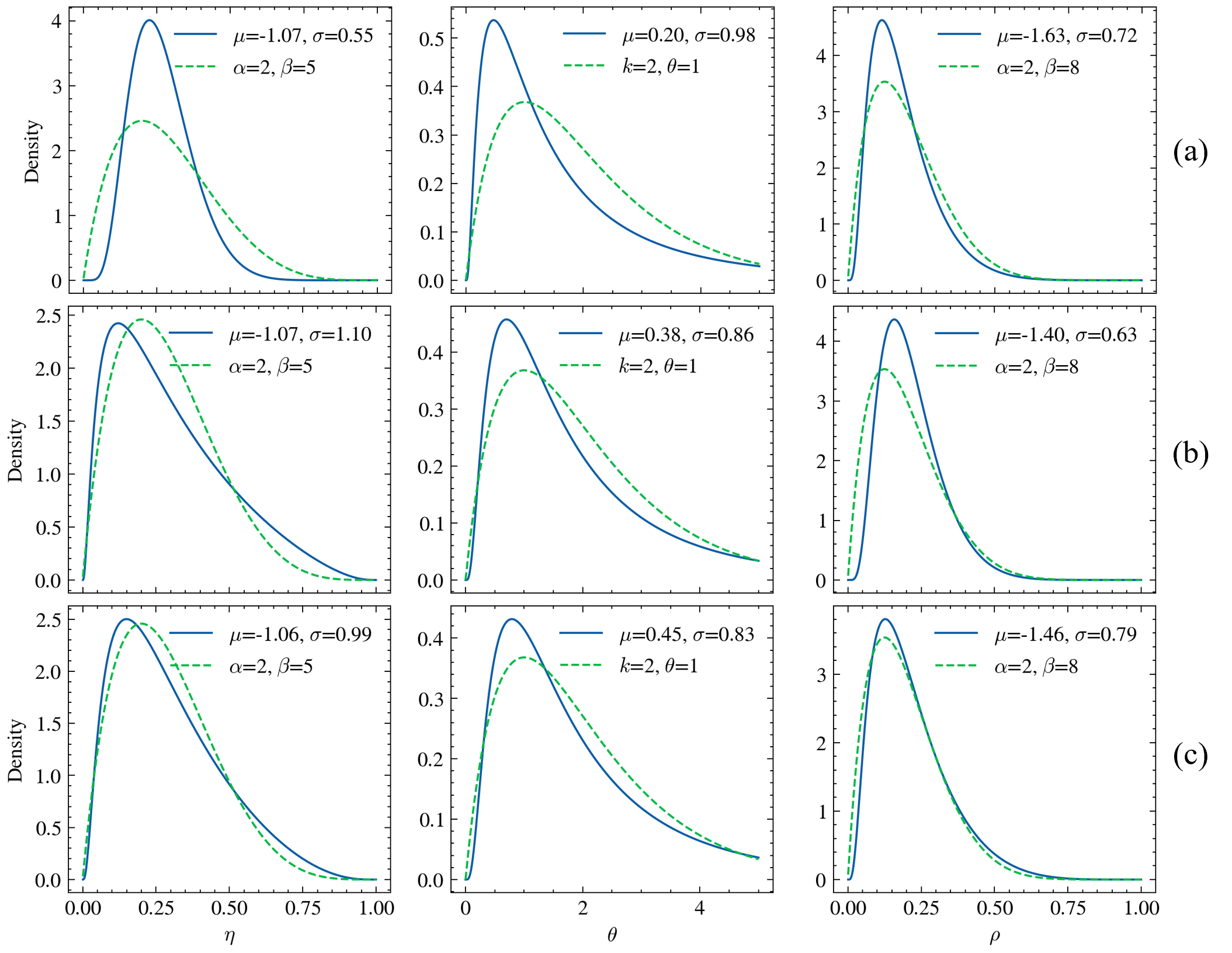}
    \caption{True population distributions versus estimated ones. (a): $N=10, T_{train}=30$; (b): $N=50, T_{train}=30$; (c): $N=50, T_{train}=80$.}
    \label{fig:hier_diff_DGP_posterior}
\end{figure}

\subsubsection{Anonymized observability}
\label{sss:hier_partial}
Because we have shown that anonymized observability does not affect the pooled model in Section \ref{sssec:partial+pool}, we focus here on the hierarchical model.

Table \ref{tb:hier_partial} reports the true values of hyperparameters and their estimates under complete versus anonymized observability. While posterior means for location parameters, such as $\mu_\eta$, are similar across settings, anonymized observability substantially underestimates dispersion (e.g., $\sigma_\eta)$. Figure \ref{fig:hier_partial} illustrates this effect using individual $\eta^n$ estimates: anonymized observability yields much less variation across individuals.

The underlying reason is the \blue{limited information about individual heterogeneity contained in aggregate count data.} Individual heterogeneity is inherently permutation-sensitive, whereas aggregate counts are permutation-invariant. As a result, groups with different within-population dispersions can generate similar aggregate behavior when their mean choice tendencies are similar. This suggests that anonymized observable data may be sufficient for predicting aggregate flows, but it provides limited behavioral insight into population-level parameter distributions.

\begin{table}[!ht]
\centering
\begin{tabular}{lllllll}
\hline
        & $\mu_\eta$ & $\sigma_\eta$ & $\mu_\theta$ & $\sigma_\theta$ & $\mu_\rho$ & $\sigma_\rho$ \\ \hline
True    & -$1.5$                                    & $0.5 $                                       & 0                                         & 1.0                                          & -2.0                                    & 1.0                                       \\
Complete observability    & -1.5                                    & 0.48                                       & 0.33                                      & 0.84                                         & -1.7                                    & 0.74                                       \\
Anonymized observability & -1.5                                    & 0.22                                       & 0.31                                      & 0.31                                         & -1.6                                    & 0.24                                       \\ \hline
\end{tabular}
\caption{Estimation of hyperparameters}
\label{tb:hier_partial}
\end{table}

\begin{figure}[!ht]
    \centering
    \includegraphics[width=0.55\linewidth]{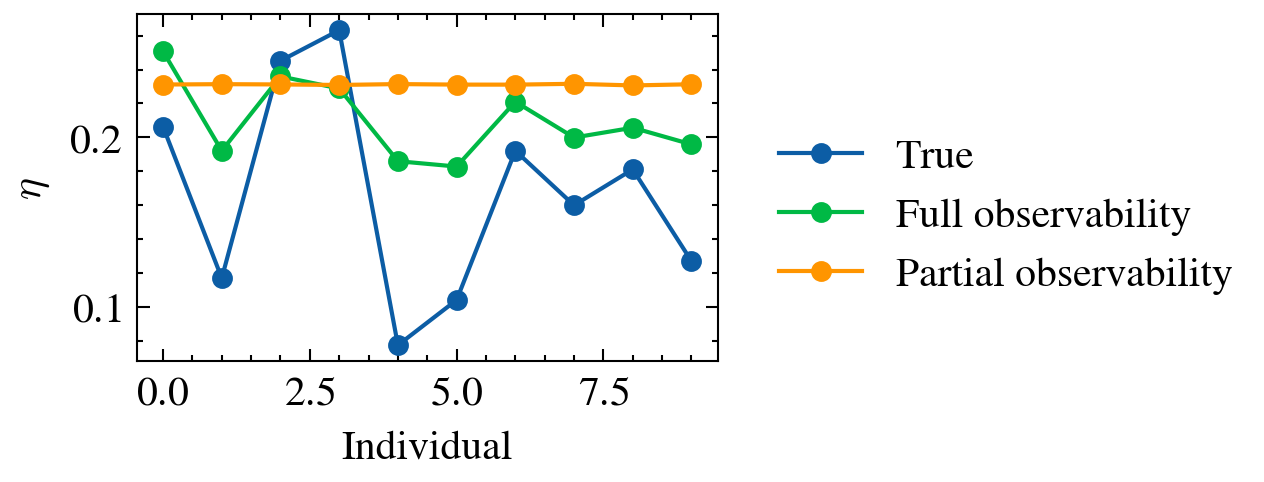}
    \caption{Estimation of individual $\eta^n$}
    \label{fig:hier_partial}
\end{figure}

\section{Empirical Analysis}
\label{sec:routing}
In this section, we examine the empirical performance of the proposed framework in both experimental and real-world settings.

\subsection{Controlled Laboratory Experiments}
We first study two controlled laboratory experiments, which allow us to assess model adequacy under different information regimes and to compare latent behavioral parameters across participant types.

\subsubsection{A study on information provision}

The first route-choice experiment is based on \citet{wijayaratna2017experimental}, illustrated in Figure~\ref{fig:exp}(a). The network follows a Braess-type structure, except that one link is stochastic: link C--B takes cost 20 with probability 20\% and cost 1 with probability 80\%. Participants complete two conditions. The first is a no-information condition, in which route choices are committed at departure, and commuters have to rely on their own experiences. The second setting incorporates information provision, in which the realization of link C--B is revealed at node C and travelers may revise their decisions.

\begin{figure}[!ht]
    \centering
    \includegraphics[width=0.7\linewidth]{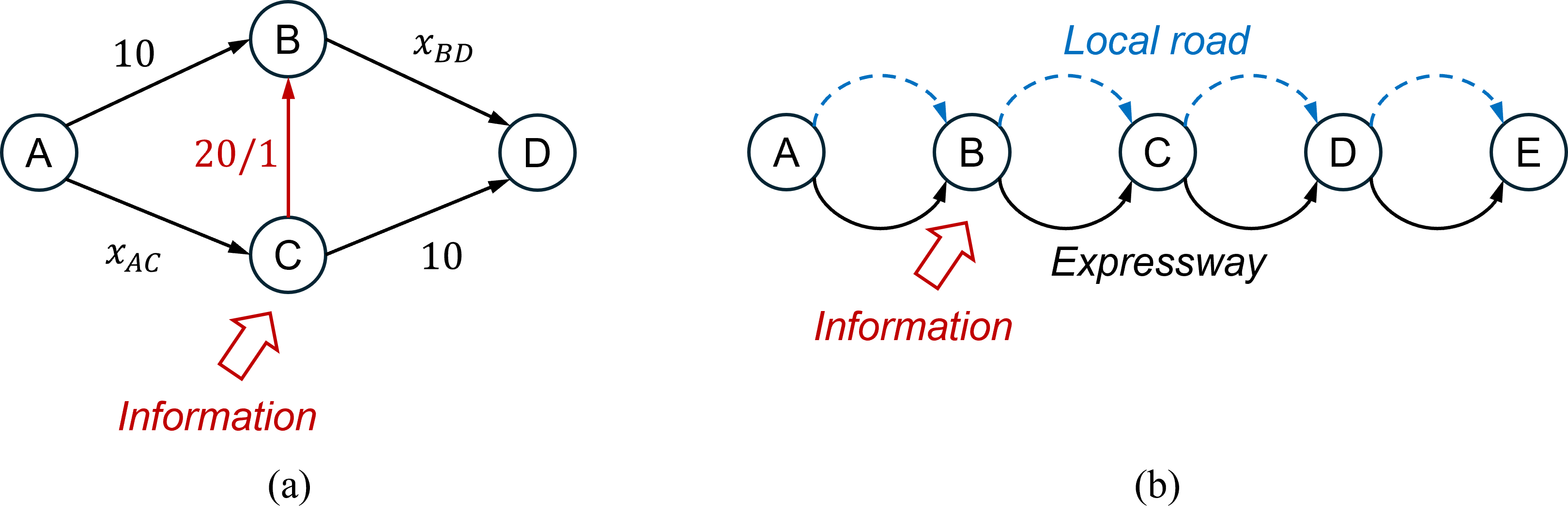}
    \caption{(a): No and en-route information. (b): Experiments with different participant types. Another OD pair faces a similar choice and shares the same expressway. We do not plot it for clarity.}
    \label{fig:exp}
\end{figure}

We begin with the no-information condition and fit the proposed model with initial perceived values all set to 0. Unlike in the simulation study, there is no ground-truth choice probability available in the laboratory data, so model assessment must rely on the realized sample path. For this reason, we focus on in-sample posterior predictive performance rather than extrapolation as in Section \ref{sec:simulation}. Specifically, for each posterior draw, we generate 500 replicated path counts over the experimental horizon and summarize the predictive distribution using the Monte Carlo mean together with 50\% and 95\% predictive intervals.

Figure~\ref{fig:wij_noinfo} in Appendix \ref{app:ss:lab} shows that, under the no-information condition, the model captures most of the observed day-to-day variation in path counts. Although finite-sample fluctuations remain substantial, the realized frequencies are generally well covered by the predictive bands. This suggests that, when travelers commit to a route, the proposed day-to-day learning model provides a reasonable explanation of the observed variation.

We then endogenize the initial perceived values, and the result is shown in Figure \ref{fig:fit_noinfo_True}. This modification improves the posterior predictive fit, suggesting an interesting insight: participants may enter the experiment with nontrivial prior perceptions rather than a completely undifferentiated initial view of the network. Moreover, the path containing the stochastic link (ACBD) tends to receive a higher estimated initial perceived cost by 1.42. While the model does not separately parameterize risk attitudes, this pattern is consistent with the possibility that participants initially discount alternatives involving uncertain costs.

\begin{figure}[!ht]
    \centering
    \includegraphics[width=0.7\linewidth]{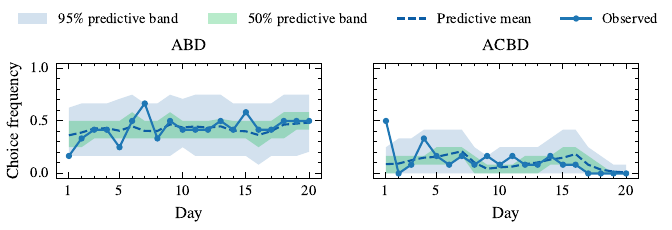}
    \caption{Posterior predictive performance for no information and endogenized initial values}
    \label{fig:fit_noinfo_True}
\end{figure}

We next fit the same class of model to the en-route information condition. As shown in Figure~\ref{fig:wij_info}, the fit deteriorates substantially: several features of the observed path-count variation are not captured by the posterior predictive distribution, especially for path ACBD when the stochastic link is realized at a high cost. By contrast, the predictive performance for path ABD remains reasonable, since this path is not meaningfully affected by the en-route information. The natural interpretation is that information provision alters the underlying behavioral mechanism in a way that is not well captured by a purely day-to-day updating framework. Adequately representing behavior under information provision may therefore require a richer model that incorporates both day-to-day and intra-day learning.

\begin{figure}[!ht]
    \centering
    \includegraphics[width=0.85\linewidth]{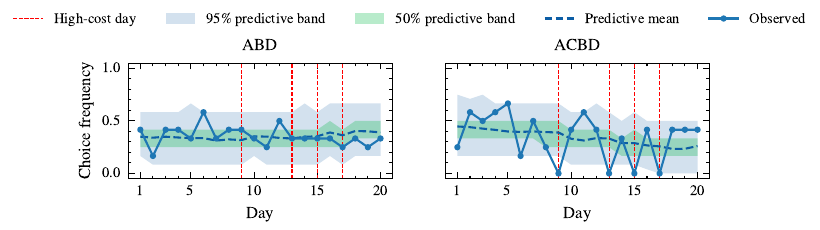}
    \caption{Posterior predictive performance for information provision and endogenized initial values}
    \label{fig:wij_info}
\end{figure}

\subsubsection{A study on behaviors across participant types}
We next turn to the experiment of \citet{wang2025comparing}, illustrated in Figure~\ref{fig:exp}(b). \blue{Their design provides a valuable benchmark for studying behavioral differences across participant types, as it evaluates humans, large language models, and reinforcement learning (RL) agents within a common sequential decision environment.} In this \blue{novel experiment}, participants, including human participants, \blue{GPT-4o}, GPT-3.5, and RL agents, traverse a route consisting of four successive segments. On each segment, they choose between two alternatives, and after completing each segment they receive congestion information from the preceding link. This creates a richer adaptive setting with repeated feedback and within-trip adjustment opportunities. \blue{We use this experiment primarily to demonstrate the inferential capability of the proposed framework. All conclusions are specific to this setting and should not be interpreted as broader claims about whether LLMs exhibit human-like learning behavior.}

In the main analysis, we represent these repeated decisions at the finer segment level, which better matches the feedback structure of the experiment. A comparison with a coarser path-level representation is reported in Appendix~\ref{app:ss:lab}. We fit the model separately to datasets generated by different participant types. Since initial values are less sensitive for longer horizons, we fix them in these experiments.

To compare behavioral differences across participant types, we adopt the Region of Practical Equivalence (ROPE), a well-established Bayesian hypothesis testing approach designed to assess whether an effect is large enough to be practically meaningful \citep{kruschke2010bayesian, kruschke2014doing}. Because the parameter $\theta$ is scale-dependent and thus less directly comparable across settings, we focus our inference on whether the learning rates are meaningfully different. A key challenge is that $\eta$ does not have a linear effect on behavioral outcomes. For example, increasing $\eta$ from 0.05 to 0.1 doubles the updating weight on new information, whereas increasing $\eta$ from 0.55 to 0.6 is comparatively minor. Following \cite{kruschke2018rejecting}, we instead test on the logit scale. For example, when comparing human and \blue{GPT-4o}, we define $\phi=\logit(\eta_{human})-\logit(\eta_{GPT4})$, and set ROPE to $[-0.1, 0.1]$. Equivalently,
\begin{equation}
    \exp(\phi)=\frac{\eta_{human}/(1-\eta_{human})}{\eta_{GPT4}/(1-\eta_{GPT4})} \in [0.905, 1.105].
\end{equation}
Intuitively, $\eta/(1-\eta)$ reflects the relative weight placed on today’s experience versus prior beliefs under exponential smoothing. Hence, this ROPE corresponds to treating differences of roughly $\pm10\%$ in the new-versus-old weighting as practically negligible. 

Figure~\ref{fig:agent_compare} summarizes the corresponding posterior contrasts. The results reveal systematic differences across groups. In panel (a), which compares humans with \blue{GPT-4o, more than 97\% of the posterior mass lies below $-0.1$. This provides strong posterior evidence that GPT-4o has a meaningfully higher learning rate than human participants, indicating greater sensitivity to recently experienced costs.} In panel (b), the posterior contrast with GPT-3.5 is shifted in the opposite direction, suggesting that humans exhibit a higher learning rate than GPT-3.5. By contrast, the posterior in panel (c), which compares humans with reinforcement-learning agents, is fairly centered around 0. However, only 2\% of the posterior mass lies within the ROPE, so the two groups cannot be regarded as practically equivalent. At the same time, the posterior does not concentrate strongly on either side, so the \blue{direction of} difference between humans and reinforcement-learning agents is not clearly distinguishable under the current data.

\begin{figure}[!ht]
    \centering
    \includegraphics[width=0.8\linewidth]{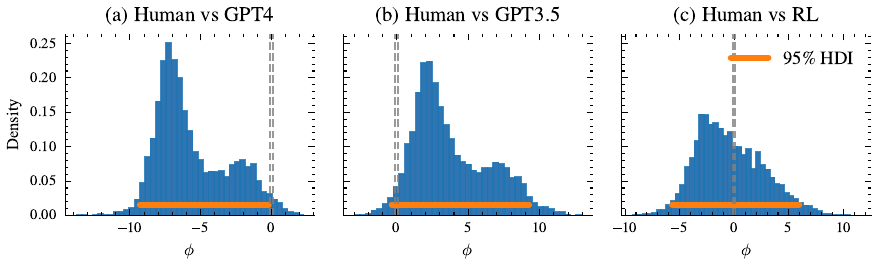}
    \caption{Posterior of $\phi=\logit(\eta_{human})-\logit(\eta_{participant})$, where $participant$ refers to GPT-4o, GPT-3.5, and RL agents}
    \label{fig:agent_compare}
\end{figure}

These differences are also reflected in the posterior predictive performance shown in Figure~\ref{fig:rep_freq_comparison} in Appendix~\ref{app:ss:lab}, where GPT-4o exhibits substantially greater volatility than the other three participant types, and are broadly consistent with the findings in \cite{wang2025comparing}. \blue{Our analysis provides an additional inferential perspective by placing all participant types within a common behavioral model and comparing an interpretable latent adjustment parameter. In this way, the proposed framework provides a principled basis for assessing whether observed differences in choice dynamics correspond to meaningful differences in underlying behavioral parameters.}

\subsection{Real-world Trajectory Data}
We now turn to trajectory data from Ann Arbor, Michigan, to examine commuter behavior in a real-world setting.

\subsubsection{Dataset and preprocessing}


The Ann Arbor network has 632 road junctions and 1,583 road segments. The road features and network skeleton were sourced from a benchmark trip-based travel demand forecast model developed for the {SEMCOG} 2050 Regional Forecast \citep{semcog_regional_forecast}. Vehicle trajectory data were collected during the evening peak (4:30--5:30 PM) from March 1 to May 31, 2022, capturing stable travel patterns within the academic semester. Following the map-matching methodology in \cite{wang2023trajectory}, the sample yielded an average of 3,127 daily trips, which were utilized to calculate link features (such as link travel time). To identify main traffic patterns, we included OD pairs observed for at least 14 days. 

From all valid OD pairs in the dataset, we select three pairs that feature meaningfully distinct routes. We exclude OD pairs whose routes differ only marginally, such as two nearly symmetric ways to traverse a square, where observed choices are largely driven by noise or improvisation rather than systematic learning. Figure \ref{fig:od_exp} illustrates one of the selected OD pairs; the remaining two are presented in Appendix \ref{app:ss:traj}.

In the OD pair shown, commuters travel from the university hospital area to the interchange of highways M-14 and I-94 when exiting Ann Arbor. Two clearly differentiated routes are available: a highway route (red) and a local-road route through downtown (blue). Summary statistics for this OD pair are reported in Table \ref{tb:od_exp}. The highway route has a lower average travel time, with slightly higher variability, and is chosen more frequently. Over the 66-day observation period, at most 10 commuters per day are observed for this OD pair. We therefore fix the total daily demand at 10 for this OD pair, as the demand parameter $\rho$ would not be identifiable without this normalization.

\begin{figure}[!ht]
    \centering
    \includegraphics[width=0.7\linewidth]{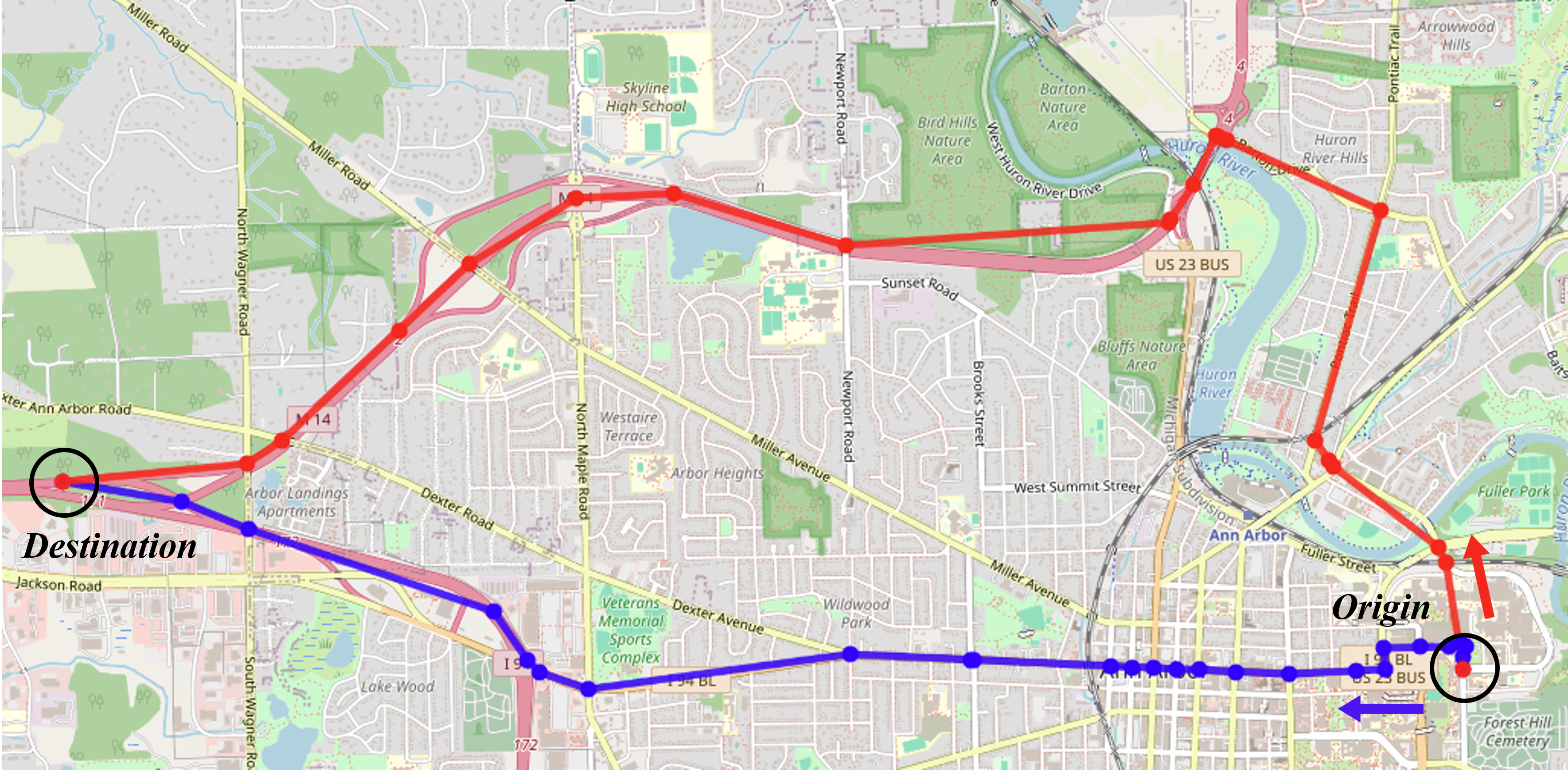}
    \caption{One of the selected OD pairs}
    \label{fig:od_exp}
\end{figure}

\begin{table}[!ht]
\centering
\begin{tabular}{ccccc}
\hline
     & \multicolumn{2}{c}{Local} & \multicolumn{2}{c}{Highway} \\
     & Cost (min)      & Counts       & Cost (min)        & Counts        \\ \hline
Mean &  11.90          &  0.94            & 11.02            &  2.06             \\
Std  &  0.98          &  1.10            &  1.00           & 1.63              \\ \hline
\end{tabular}
\caption{Statistics of the selected OD pair}
\label{tb:od_exp}
\end{table}

\subsubsection{Pooled model esitmation}

We first estimate the pooled model to recover population-average behavioral parameters. The results are reported in Table \ref{tb:estimates}, where $\delta_{od}$ denotes the initial value difference for OD pair $od=0,1,2$.

One notable feature of real-world data is substantial day-to-day demand variation. In this dataset, each commuter has more than a 60\% probability of not traveling on a given weekday. Such flexibility, likely influenced by remote work, is rarely present in lab experiments but is central in modern commuting behavior. The estimated negative value of $\delta_0$ indicates that commuters initially perceive the highway route as more attractive than the local route for the OD pair in Figure \ref{fig:od_exp}, which is broadly consistent with the observed cost statistics. 

\begin{table}[!ht]
\centering
\begin{tabular}{ccccccc}
\hline
          & $\eta$ & $\theta$ & $\rho$ & $\delta_0$ & $\delta_1$ & $\delta_2$ \\ \hline
Estimates & 0.024               & 0.39                  & 0.61                & -2.4                         & 0.066                        & 2.1                          \\ \hline
\end{tabular}
\caption{Estimation results}
\label{tb:estimates}
\end{table}

Given this result, it is natural to ask whether this day-to-day learning is practically meaningful. We therefore test whether the estimated $\eta$ is meaningfully different from 0. Because $\logit(0)$ is not well-defined, we assess practical equivalence of $\eta$ directly using the ROPE $[0,\epsilon]$. We choose $\epsilon=0.693/66$, which corresponds approximately to an half-life of exponential smoothing over 66 days, the length of the observation period. This is a fairly large threshold for practical equivalence: it treats as negligible even learning rates for which the initial belief is reduced by half over the course of the study. Figure \ref{fig:delta_eta} shows that more than 83\% of the posterior mass lies above this ROPE. \blue{Thus, most of the posterior support favors a learning rate that exceeds this practical-equivalence threshold, providing evidence that day-to-day updating is behaviorally relevant in the observed commuting data.}
\begin{figure}[!ht]
    \centering
    \includegraphics[width=0.35\linewidth]{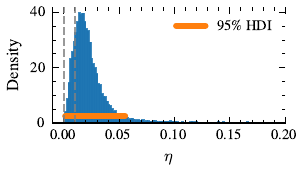}
    \caption{Empirical posterior of real-world commuters' learning rate}
    \label{fig:delta_eta}
\end{figure}

\subsubsection{Hierarchical model and behavior dispersion}
Although we showed in Section \ref{sss:hier_partial} that anonymized observability weakens identification of heterogeneity, we nevertheless estimate the hierarchical model to gain qualitative insight into the distribution of behaviors across the population. Figure \ref{fig:hierarchical_GM_distributions} reports the implied population distributions of individual parameters based on the estimated hyperparameters. As discussed earlier, dispersion may be underestimated due to anonymized observability. Key findings include:
\begin{itemize}
\item Learning rate $\eta$: highly concentrated between 0 and 0.1. Nearly all commuters place less than 10\% weight on new daily information relative to accumulated beliefs. It indicates that some proportion of commuters are doing meaningful day-to-day updating, but strong inertia persists and makes updates very gradual.
\item Costs sensitivity $\theta$: moderately dispersed, with most mass between 0.1 and 0.7. This indicates meaningful variation in responsiveness to perceived cost differences. 
\item Demand parameter $\rho$: exhibits the widest dispersion. Only a small fraction commute nearly every weekday, while more than half of commuters stay home at least half of the time. This highlights substantial flexibility and the importance of modeling demand variability explicitly when analyzing real-world commuter behavior.
\end{itemize}

\begin{figure}[!ht]
    \centering
    \includegraphics[width=0.65\linewidth]{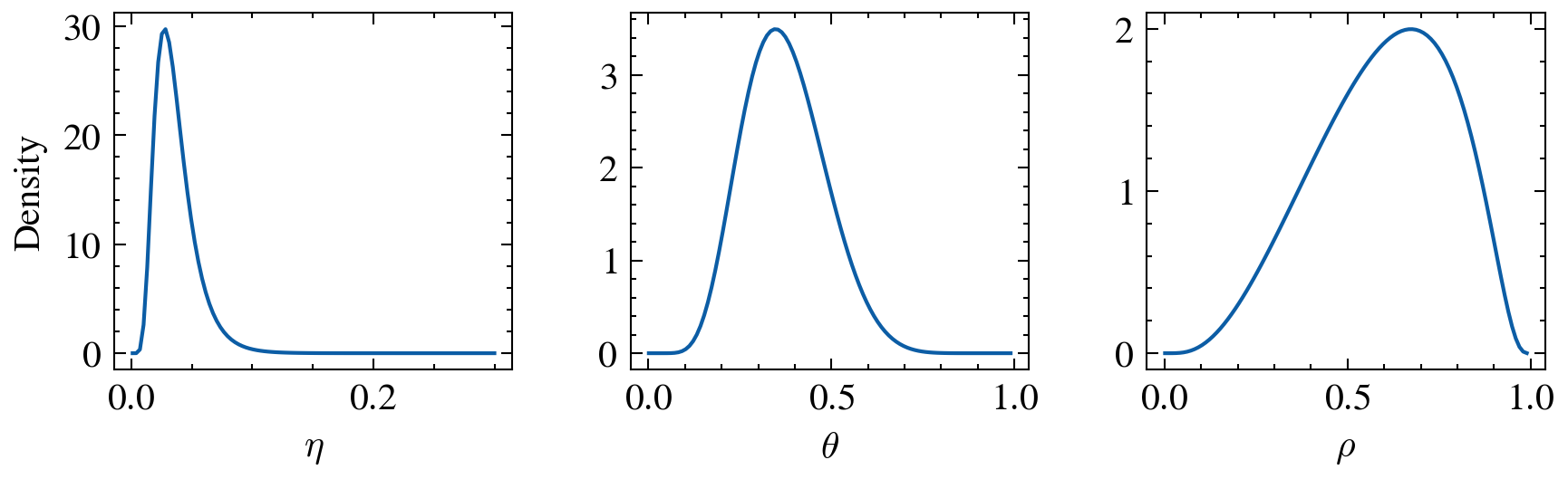}
    \caption{Population distributions of individual parameters}
    \label{fig:hierarchical_GM_distributions}
\end{figure}

\section{Conclusion}\label{sec:conclusion}
This paper develops a Bayesian inference framework for day-to-day route choice behavior that moves beyond descriptive calibration to support uncertainty quantification, hypothesis testing, and rigorous behavioral analysis. The central theoretical contributions are formal identifiability and consistency results that clarify when and how behavioral parameters can be recovered from observations, properties that prior calibration-based approaches neither establish nor require. The framework accommodates demand variation, user heterogeneity through a hierarchical extension, and anonymized observability motivated by privacy constraints on trajectory data.

Simulation studies confirm that the estimator achieves good finite-sample performance: bias shrinks systematically as the observation horizon grows, and credible intervals are well-calibrated across a range of parameter configurations. Importantly, robustness analyses show that aggregate flow predictions and population distribution estimates remain stable under moderate model misspecification. Hence, the framework degrades gracefully when behavioral assumptions are imperfectly met, which is a necessary condition for practical applicability. 

The empirical analysis shows that the proposed framework can generate novel behavioral insights from both experimental and real-world data. In controlled laboratory experiments, it reveals that the day-to-day learning model explains much of the observed variation under no-information conditions. Once en-route information is introduced, however, the same model becomes inadequate, indicating that richer information changes the decision process in ways not well captured by a purely inter-day updating mechanism. A second laboratory application recovers systematic behavioral differences across participant types: humans appear less reactive than \blue{GPT-4o and more reactive than GPT-3.5, while the comparison with reinforcement-learning agents does not yield a clearly directional difference.}  In the real-world Ann Arbor data, the results provide substantial evidence of meaningful day-to-day learning, while also revealing considerable demand variations.

\blue{The behavioral inference framework can also provide actionable insights for transportation system management. For example, the estimated parameters can serve as behavioral inputs to adaptive congestion-pricing or route-guidance policies. In particular, the relatively low learning rates estimated from the Ann Arbor data suggest that the effects of a new toll or information intervention may emerge gradually over multiple days. Evaluating such an intervention only during its initial implementation period could therefore understate its longer-term behavioral effects.}


Several directions remain for future research. First, the inference structure developed here, exponential smoothing of perceived costs combined with a logit choice rule, is not specific to route choice. With appropriate reformulation, the same framework can be applied to other behavioral dynamics, such as departure time adjustment or mode choice adaptation with appropriate reformulation. Second, while the hierarchical model provides qualitative insight into heterogeneity, stronger identification of population distributions under anonymized observability remains an open theoretical problem. It may require richer data sources, alternative model designs, or a full mixture identifiability argument that we have left for future work. \blue{Finally, extending the inferential framework to explicitly model day-to-day behavior in the presence of navigation and real-time information systems would further bridge the gap between theoretical day-to-day dynamics and contemporary travel behavior.}

%
%
%
\section*{Acknowledgement}
The work described in this paper was partly supported by research grants from National Science Foundation, United States (CMMI-2233057 and 2240981). The authors thank Xinhe Wang, Naichen Shi, and Ran Sun for valuable feedback.

\bibliography{reference}

\section*{Appendix A. Proof of Main Results}
\label{app:proof}
\subsection*{Proof for Proposition \ref{pp:horowitz}.}
As the valuations $V_t^n(i)$ follow the exact same updating process as in Horowitz dynamics, we only need to consider the choice probability, which results in path flow following the observation $O_t$, detailed in Section \ref{sssec:partial+pool}. 

As discussed in that section, $O_t=(O_t(1),...,O_t(M))\sim \text{Multinomial}(N,p_t)$, its expectation follows:
\begin{equation}
    \scrE[O_t] = N p_t = N \cdot  softmax(V_t),
\end{equation}
which matches the path flow in Horowitz dynamics.

The path flow observation satisfies
\begin{equation}
    O_t(i) = \sum_{n=1}^N \textbf{1}\left\{X_t^n=i\right\},
\end{equation}
and $X_t^n$ is i.i.d categorical with $p_t$. Thus, by the strong law of large numbers, 
\begin{equation}
    \frac{1}{N} O_t(i) \asto p_t(i), \ \forall i\in[M],
\end{equation}
which proves the proposition. $\square$

\subsection*{Proof for Theorem \ref{thm:basic}.}
First, the following lemma indicates that we only need to prove identifiability for a system with only one commuter.
\begin{lemma}
\label{lm:one-for-all}
If for any $n\in[N]$, the model is identifiable if 
\begin{equation}
\label{eq:thm1-single-match}
    P(X^n_{1:T}=x^n_{1:T}|\eta, \theta, \rho,c_\fullT) = P(X_{1:T}^n=x^n_{1:T}|\eta', \theta', \rho',c_\fullT), \ \forall x_{1:T}^n\in [M]^T 
\end{equation}
indicates $\eta=\eta',\theta=\theta',\rho=\rho'$.

\end{lemma}

\textit{Proof for Lemmea \ref{lm:one-for-all}.}

Suppose we have the identifiability for a model with only one commuter. When the following equation holds:
\begin{equation}
    P(X_{1:T}=x_{1:T}|\eta, \theta,\rho,c_{1:T}) = P(X_{1:T}=x_{1:T}|\eta', \theta',\rho',c_{1:T}), \ \forall x_{1:T}\in[M]^{T\times N},
\end{equation}
it must also hold for the case where all commuters replicate the choice of commuter $n$. Therefore, 
\begin{equation}
    \prod_{j=1}^NP(X^j_{1:T}=x^n_{1:T}|\eta, \theta, \rho,c_\fullT) = \prod_{j=1}^NP(X^j_{1:T}=x^n_{1:T}|\eta', \theta', \rho',c_\fullT), \ \forall x_{1:T}^n\in [M]^T.
\end{equation}

In the pooled model, all commuters share the same choice distribution, hence this equation implies that Equation (\ref{eq:thm1-single-match}) also holds. Therefore, $\eta=\eta',\theta=\theta',\rho=\rho'$. $\blacksquare$

Let us now consider the system with $N=1$ and omit the commuter superscript.
Suppose that two parameter sets $(\eta,\theta,\rho)$ and
$(\eta',\theta',\rho')$ are observationally equivalent. Denote the perceived cost under the two parameter sets as $\{V_t(i)\}_{t\in[T], i\in[M]}$ and $\{V'_t(i)\}_{t\in[T], i\in[M]}$, respectively. Equality of the
non-travel probabilities immediately implies $\rho=\rho'$.
Moreover, equality of the route-choice probabilities implies equality of the
corresponding log-odds:
\begin{equation}
\label{eq:master}
    \theta\Delta V_t(i,j)
    =
    \theta'\Delta V_t'(i,j),
    \qquad
    t=1,\ldots,T,\quad i,j\in[M].
\end{equation}

Define the first informative day as
\begin{equation}
    K=\min\left\{
    t\in\{1,\ldots,T-2\}:
    \Delta c_t(i,j)\neq\Delta V_1(i,j)
    \text{ for some }i,j\in[M]
    \right\}.
\end{equation}
Assumption~\ref{ass:rich} ensures that $K$ exists. If $K=1$, $\Delta V_K(i,j) = \Delta V_K'(i,j)=\Delta V_1(i,j)$ for all route pairs. Otherwise, by the definition of $K$,
\begin{equation}
    \Delta c_t(i,j)=\Delta V_1(i,j),
    \qquad t<K,\quad i,j\in[M],
\end{equation}
which also implies
    $\Delta V_K(i,j)=\Delta V_K'(i,j)=\Delta V_1(i,j)$
for all route pairs based on the exponential learning process.

We distinguish two cases. First, suppose that
$\Delta V_1(a,b)\neq0$ for some pair $(a,b)$. Applying
Equation~(\ref{eq:master}) at $t=1$ gives $\theta=\theta'$. Choose a pair
$(m,m')$ satisfying 
  $  \Delta c_K(m,m')\neq\Delta V_1(m,m').$
For this pair, the updating equation gives
\begin{align}
    \Delta V_{K+1}(m,m')
    &=
    \Delta V_1(m,m')
    +\eta\left[
    \Delta c_K(m,m')-\Delta V_1(m,m')
    \right],\\
    \Delta V_{K+1}'(m,m')
    &=
    \Delta V_1(m,m')
    +\eta'\left[
    \Delta c_K(m,m')-\Delta V_1(m,m')
    \right].
\end{align}
Equation~(\ref{eq:master}) at $t=K+1$, together with
$\theta=\theta'$, therefore implies $\eta=\eta'$.

Second, suppose that $\Delta V_1(i,j)=0$ for all route pairs. For the
informative pair $(m,m')$, we then have
\begin{equation}
    \Delta V_{K+1}(m,m')=\eta\Delta c_K(m,m'),
    \qquad
    \Delta V_{K+1}'(m,m')=\eta'\Delta c_K(m,m').
\end{equation}
Because $\Delta c_K(m,m')\neq0$, Equation~(\ref{eq:master}) at $t=K+1$
implies 
   $ \theta\eta=\theta'\eta'.$
Denote this common positive value by $R$. On day $K+2$,
\begin{align}
    \theta\Delta V_{K+2}(m,m')
    &=
    R\left[
    (1-\eta)\Delta c_K(m,m')
    +\Delta c_{K+1}(m,m')
    \right],\\
    \theta'\Delta V_{K+2}'(m,m')
    &=
    R\left[
    (1-\eta')\Delta c_K(m,m')
    +\Delta c_{K+1}(m,m')
    \right].
\end{align}
Equation~(\ref{eq:master}) at $t=K+2$ consequently implies
\begin{equation}
    R(\eta'-\eta)\Delta c_K(m,m')=0.
\end{equation}
Since $R>0$ and $\Delta c_K(m,m')\neq0$, it follows that
$\eta=\eta'$, and hence $\theta=\theta'$.

In both cases, observational equivalence implies
$(\eta,\theta,\rho)=(\eta',\theta',\rho')$. Because $K\leq T-2$, both
$K+1$ and $K+2$ lie within the observed horizon. Therefore, the model is
identifiable under Assumption~\ref{ass:rich}. $\square$

\subsection*{Proof for Theorem \ref{thm:consistent}.}
Lemma \ref{lm:one-for-all} indicates that we only need to prove identifiability for a system with only one commuter.
Let $\psi=(\eta,\theta)$.
Denote the choice probability as 
\begin{equation}
    s_{t}(\psi, m):=\frac{e^{-\theta V_t(m;\eta,\theta)}}{\sum_{k=1}^Me^{-\theta V_t(k;\eta,\theta)}}, \ m\in[M].
\end{equation}
Define the full choice probability
\begin{equation}
    p_t(\phi,x):=\begin{cases}
        \rho, \ & x=0,\\
        (1-\rho) s_t(\psi,m), \ & x\in[M].
    \end{cases}
\end{equation}
Define a random function $p_{t}(\phi,X_t)$ for the random variable $X_t$, which takes value of $p_t(\phi,x)$ for each possible choice $x\in \{0,1,...,M\}$. Define a general log-likelihood function for the random sequence
\begin{equation}
    l_T(\phi,X_\fullT)=\sum_{t=1}^T\log p_{t}(\phi,X_t),
\end{equation}
which is also a random function.

Define $
    Z_{t}(\phi, X_t)=\log p_{t}(\phi,X_t)-\log p_{t}(\phi_0,X_t),$
which can be rewritten as 
\begin{equation}
    Z_{t}(\phi, X_t)= \scrE[Z_{t}(\phi,X_t)|\scrF_\tmo] + \underbrace{Z_{t}(\phi, X_t) - \scrE[Z_{t}(\phi,X_t)|\scrF_\tmo]}_{\text{Denoted as } D_{t}(\phi,X_t)}.
\end{equation}

The centered value can be rewritten as:
\begin{align}
    \scrE[Z_{t}(\phi,X_t)|\scrF_\tmo] &= \sum_{m=1}^M p_t(\phi_0,m) [\log p_t(\phi,m)-\log p_t(\phi_0,m)] \\
    &=-\KL (p_t(\phi_0)\Vert p_t(\phi)),
\end{align}
where $p_t(\phi)=\left\{p_{t}(\phi,x)\right\}_{x\in\{0,1,...,M\}}$ denotes the choice distribution vector. 

Define $A\sube=\left\{\phi\in\Theta:\Vert \phi-\phi_0\Vert \geq \epsilon \right\}$. 
The posterior can be written as
\begin{align}
    P(A\sube|X_\fullT)&=\frac{\int_{A\sube}e^ {l_T(\phi, X_\fullT)} p(\phi) d\phi}{\int_{\Theta}e^{l_T(\phi, X_\fullT)} p(\phi) d\phi} \\
    &= \frac{\int_{A\sube} e^{l_T(\phi,X_\fullT)-l_T(\phi_0,X_\fullT)} p(\phi) d\phi}{\int_{\Theta} e^{l_T(\phi,X_\fullT)-l_T(\phi_0,X_\fullT)} p(\phi) d\phi} \label{eq:posterior}
\end{align}
and our goal is to prove the numerator decays faster than the denominator as $T\to \infty$. 

We first prove a lemma regarding $D_{t}(\phi,X_t)$. 

\begin{lemma}
\label{lm:point_converge}
For any $\phi\in\Theta$,
\begin{equation}
    \frac{1}{T} \sum_{t=1}^T D_t(\phi, X_t) \asto 0 \quad \mathrm{as} \ T\to\infty
\end{equation}
\end{lemma}

\textit{Proof for Lemma \ref{lm:point_converge}}
First, it is easy to show that $V_t(m;\eta,\theta)$ is uniformly bounded. Let $|V_t(m;\eta,\theta)|\leq C$, then 
\begin{equation}
    s_t(\psi,m)\geq \frac{e^{-2\theta_{max}C}}{M}.
\end{equation}
Thus, the full choice probability is also uniformly bounded away from 0:
\begin{equation}
    p_t(\phi,x) \geq \min\left\{\rho_{min}, (1-\rho_{max}) \frac{e^{-2\theta_{max}C}}{M} \right\}>0.
\end{equation}
Therefore, $\log p_{t}(\phi,x)$ is also uniformly bounded, which further indicates that $Z_{t}(\phi, X_t)$ is also uniformly bounded. Denote the bound as $K$, i.e., $\vert Z_t(\phi, X_t)\vert <K$, and thus $|D_{t}(\phi,X_t)|<2K$ for all $t\in[T], \phi\in\Theta$ and $X_t\in[M]$.

As $\scrE[D_t(\phi,X_t)|\scrF_\tmo]=0$, $\{D_t(\phi,X_t),\scrF_\tmo\}$ is a martingale difference sequence. 
In addition,
\begin{equation}
    D_{t}^2(\phi,X_t)=Z_{t}(\phi,X_t)^2+\scrE[Z_{t}(\phi,X_t)|\scrF_\tmo]^2 -2 Z_{t}(\phi,X_t) \cdot \scrE[Z_{t}(\phi,X_t)|\scrF_\tmo].
\end{equation}
Its expectation satisfies:
\begin{equation}
    \scrE[D_{t}^2(\phi,X_t)|\scrF_\tmo] = \scrE[Z_{t}(\phi,X_t)^2|\scrF_\tmo] - \scrE[Z_{t}(\phi,X_t)|\scrF_\tmo]^2 \leq 2K^2.
\end{equation}
Therefore, 
\begin{equation}
    \sum_{t=1}^\infty \frac{ \scrE[D_{t}^2(\phi,X_t)|\scrF_\tmo]}{t^2}< \infty 
\end{equation}

By the strong law of martingale differences, we have
\begin{equation}
    \frac{1}{T} \sum_{t=1}^T D_{t}(\phi,X_t) \to 0, \quad a.s.
\end{equation}
as $T\to\infty$.
$\blacksquare$

To simplify notations, let us denote $G_T(\phi, X_\fullT) = \frac{1}{T}\sum_{t=1}^TD_t(\phi,X_t)$, for any $T$. This random function can be proved to be uniformly Lipschitz continuous. 
\begin{lemma}
\label{lm:lipschitz}
There exists a constant $L$, such that for any $\phi,\phi'\in \Theta$, $\vert G_T(\phi,X_\fullT)-G_T(\phi',X_\fullT)\vert \leq L\Vert\phi-\phi'\Vert$ for all $T$ and realization of $X_\fullT$.
\end{lemma}

\textit{Proof for Lemma \ref{lm:lipschitz}}
It is easy to see that the value function is uniformly bounded, i.e., $|V_t(m;\eta)|\leq C$ for all $t,m,\eta$. We further have
\begin{equation}
    \frac{\partial V_\tpo(m;\eta)}{\partial \eta} = (1-\eta) \frac{\partial V_t(m;\eta)}{\partial \eta} -V_t(m;\eta) + c_t(m), \quad \forall t,m,\eta
\end{equation}

Let us prove $\left\vert \frac{\partial V_t(m;\eta)}{\partial \eta}\right\vert\leq \frac{2C}{\eta_{min}}$ by induction on $t$. This holds for $t=1$ automatically. Suppose it also holds for $t=k$:
\begin{align}
    \left\vert\frac{\partial V_{k+1}(m;\eta)}{\partial \eta} \right\vert &\leq (1-\eta) \left\vert\frac{\partial V_t(m;\eta)}{\partial \eta}\right\vert +2C \leq \frac{2C}{\eta_{min}}.
\end{align}
Therefore, it also holds for $t=k+1$, which proves the claim.

Now, let us prove that the  $\log s_t(\psi,k)$ has a bounded gradient. The log-likelihood can be written as:
\begin{equation}
    \log s_t(\psi, k) = -\theta V_t(k;\eta) -\log \sum_{m=1}^M e^{-\theta V_t(m;\eta)}, \quad k\in[M], t\in[T].
\end{equation}

Therefore,
\begin{align}
    \frac{\partial}{\partial \theta} \log s_t(\psi, k)&=-V_t(k;\eta)-\frac{\sum_{m=1}^M e^{-\theta V_t(m;\eta)} (-V_t(m;\eta))}{\sum_{m=1}^M e^{-\theta V_t(m;\eta)}} \\
    &= -V_t(k;\eta) + \sum_{m=1}^M V_t(m;\eta) s_t(\psi,m) ,
\end{align}
which indicates that 
\begin{equation}
    \left\vert\frac{\partial}{\partial \theta} \log s_t(\psi, k) \right\vert \leq | V_t(k;\eta)| + \sum_{m=1}^M  s_t(\psi,m) |V_t(m;\eta)| \leq 2C.
\end{equation}

Meanwhile,
\begin{align}
    \frac{\partial}{\partial \eta} \log s_t(\psi, k)&=-\theta \frac{\partial}{\partial \eta}V_t(k;\eta)-\frac{\sum_{m=1}^M e^{-\theta V_t(m;\eta)}\frac{\partial}{\partial \eta}(-\theta V_t(m;\eta))}{\sum_{m=1}^M e^{-\theta V_t(m;\eta)}} \\
    &= -\theta \frac{\partial}{\partial \eta}V_t(k;\eta)+\sum_{m=1}^M s_t(\psi,m) \theta \frac{\partial}{\partial \eta}V_t(m;\eta),
\end{align}
which leads to 
\begin{equation}
    \left\vert\frac{\partial}{\partial \eta} \log s_t(\psi, k) \right\vert \leq    \frac{4C\theta_{max}}{\eta_{min}}.
\end{equation}

As a result, for the full chocie probability $p_t(\phi,x)$ where $\phi=(\eta,\theta,\rho)$,
\begin{equation}
    \nabla_\phi \log p_t(\phi, x)  = \begin{cases}
        \begin{bmatrix}
            0, & 
            0, &
            \frac{1}{\rho} 
        \end{bmatrix}^T, \ & x=0,\\
        \begin{bmatrix}
        \frac{\partial}{\partial \eta} \log s_t(\phi,k), &
        \frac{\partial}{\partial \theta} \log s_t(\phi,k), &
        -\frac{1}{1-\rho}
    \end{bmatrix}^T, \ & x\in[M].
    \end{cases} 
\end{equation}
Let $B_\rho = \max \{1/\rho_{min}, 1/(1-\rho_{max})\}$, we have
    $\Vert \nabla_\phi \log p_t(\phi,x)\Vert \leq C_L$,
where 
\begin{equation}
    C_L = \sqrt{(2C)^2 + \left({4C\theta_{max}}/{\eta_{min}}\right)^2 + B_\rho^2},
\end{equation}
which further indicates that $\Vert \nabla_\phi Z_t(\phi,m)\Vert \leq C_L$ for all $t\in[T], x\in\{0,1,...,T\}$. 

As $\Theta$ is convex, based on the Mean Value Theorem, for any fixed realization of $X_t=x$, 
\begin{equation}
    \vert Z_t(\phi,x)-Z_t(\phi',x)\vert \leq C_L \Vert\phi-\phi'\Vert, \ \forall \phi,\phi'\in\Theta.
\end{equation}

Because $X_t$ takes vales in a finite set, and the gradient is uniformly bounded for any $x$, we have
\begin{equation}
    \vert Z_t(\phi,X_t)-Z_t(\phi',X_t) \vert \leq C_L\Vert\phi-\phi'\Vert, \ \forall \phi,\phi'\in\Theta.
\end{equation}

Meanwhile, 
\begin{align}
    |\scrE[Z_t(\phi,X_t)|\scrF_\tmo]-\scrE[Z_t(\phi',X_t)|\scrF_\tmo] | & = \left\vert \sum_{x=0}^M Z_t(\phi,x) p_t(\phi_0,x) -\sum_{x=0}^M Z_t(\phi',x)p_t(\phi_0,x) \right\vert \\
    &\leq C_L\Vert \phi-\phi'\Vert , \quad \forall \phi,\phi'\in \Theta,
\end{align}
which indicates $\vert D_t(\phi,X_t)-D_t(\phi',X_t)\vert \leq 2C_L\Vert \phi-\phi'\Vert$.

Finally,
\begin{equation}
    \vert G_T(\phi, X_\fullT)-G_T(\phi', X_\fullT)\vert \leq \frac{1}{T} \sum_{t=1}^T \vert D_t(\phi, X_t)-D_t(\phi', X_t)\vert \leq 2C_L\Vert\phi-\phi'\Vert, \quad \forall \phi,\phi'\in\Theta. \ \blacksquare
\end{equation}

The Lipschitz continuity further ensures that the convergence is not only pointwise as in Lemma \ref{lm:point_converge} but also uniform.
\begin{lemma}
\label{lm:uniform_converge}
\begin{equation}
    \sup_{\phi\in\Theta} \left\vert G_T(\phi, X_\fullT) \right\vert\asto 0 \quad \mathrm{as} \ T\to\infty.
\end{equation}
\end{lemma}

\textit{Proof for Lemma \ref{lm:uniform_converge}}
Based on the Lipschitz continuity in Lemma \ref{lm:lipschitz}, for any $\epsilon>0$, there exists $\delta$ such that for all $\phi,\phi'$ with $\Vert \phi-\phi'\Vert<\delta$, $|G_T(\phi, X_\fullT)-G_T(\phi',X_\fullT)|<\epsilon/2$.

Since $\Theta$ is compact, it is separable, which means that it can be covered with a finite number of open balls of radius $\delta$. Denote the centers as $\left\{\phi_1,...,\phi_K\right\}$. Thus, for any $\phi\in \Theta$, it is within distance $\delta$ from some center $\phi_j$.

For any $i\in[K]$, from Lemma \ref{lm:point_converge}, $G_T(\phi_i, X_\fullT)\asto 0$ as $T\to\infty$. Therefore, there exists $T_i$ such that for all $T\geq T_i$, $|G_T(\phi_i, X_\fullT)|<\epsilon/2$ almost surely. Let $\Tilde{T} = \max_{i\in[K]} T_i$, thus for all $T\geq \Tilde{T}$, $|G_T(\phi_j, X_\fullT)|<\epsilon/2$ almost surely for all $j\in[K]$. 

Consequently, for any $\phi\in\Theta$, find closest center $\phi_j$ where $\Vert \phi-\phi_j\Vert<\delta$, then for all $T\geq \Tilde{T}$:
\begin{align}
    |G_T(\phi, X_\fullT)| &\leq |G_T(\phi,X_\fullT)-G_T(\phi_j,X_\fullT)| + |G_T(\phi_j, X_\fullT)| \leq \epsilon  \quad \mathrm{a.s.}
\end{align}

Note that $\Tilde{T}$ is independent of $\phi$, therefore
\begin{equation}
    \sup_{\phi\in\Theta} |G_T(\phi,X_\fullT)|<\epsilon \ \text{ for } T\geq \Tilde{T} \quad \mathrm{a.s.},
\end{equation}
which proves the uniform convergence. 
$\blacksquare$

Now we are ready to prove the posterior consistency.
\begin{lemma}
\label{lm:posterior_consistent}
Under conditions in Theorem \ref{thm:consistent}, for all $\epsilon>0$, 
\begin{equation}
    P(A\sube|X_\fullT)\asto 0 \quad \text{as } T\to\infty.
\end{equation}
\end{lemma}

\textit{Proof for Lemma \ref{lm:posterior_consistent}.}
For any $\epsilon>0$, the numerator in Equation (\ref{eq:posterior}) satisfies:
\begin{align}
    \sup_{\phi\in A\sube} \frac{1}{T}(l_T(\phi,X_\fullT)- &l_T(\phi_0,X_\fullT)) = \sup_{\phi\in A\sube} \left( -\frac{1}{T} \sum_{t=1}^T \KL(p_t(\phi_0)\Vert p_t(\phi)) +\frac{1}{T} \sum_{t=1}^TD_t(\phi,X_t)\right) \\
    &\leq \underbrace{\sup_{\phi\in A\sube} \left( -\frac{1}{T} \sum_{t=1}^T \KL(p_t(\phi_0)\Vert p_t(\phi))\right)}_{\leq -\kappa(\epsilon) \ \text{by Assumption \ref{ass:persist}}} + \underbrace{\sup_{\phi\in A\sube}\left(\frac{1}{T} \sum_{t=1}^TD_t(\phi,X_t)\right) }_{\leq \kappa(\epsilon)/2 \ \text{almost surely for sufficiently large $T$ by Lemma \ref{lm:uniform_converge}}} \\
    &\leq -\frac{\kappa(\epsilon)}{2} \quad \mathrm{a.s.} \quad \text{for sufficiently large } T
\end{align}
Hence, for sufficiently large $T$
\begin{equation}
    \sup_{\phi\in A\sube} (l_T(\phi, X_\fullT)-l_T(\phi_0,X_\fullT)) \leq -\frac{\kappa(\epsilon)}{2}T \quad \mathrm{a.s.}
\end{equation}

Consequently, the numerator satisfies
\begin{equation}
    \int_{A_\epsilon} e^{l_T(\phi, X_\fullT)-l_T(\phi_0, X_\fullT)}p(\phi)d\phi \leq \int_{A_\epsilon} e^{-T\kappa(\epsilon)/2}p(\phi)d\phi \leq e^{-T\kappa(\epsilon)/2} \quad \text{a.s.}
\end{equation}

For the denominator, as we have already shown $\vert Z_t(\phi,X_t)-Z_t(\phi',X_t) \vert \leq C_L \Vert\phi-\phi'\Vert, \ \forall \phi,\phi'\in\Theta$ in Lemma \ref{lm:lipschitz},
\begin{equation}
    |\log p_t(\phi,X_t) - \log p_t(\phi_0,X_t) | \leq C_L\Vert \phi-\phi_0\Vert.
\end{equation}
Thus,
\begin{equation}
    |l_T(\phi,X_\fullT)-l_T(\phi_0,X_\fullT)| \leq \sum_{t=1}^T |\log p_t(\phi,X_t) - \log p_t(\phi_0,X_t)| \leq C_LT\Vert\phi-\phi_0\Vert .
\end{equation}

Define a small ball $B_T:=\left\{\phi\in\Theta: \Vert\phi-\phi_0\Vert \leq 1/\sqrt{T}\right\}$, whose size decreases with $T$. Thus, it is easy to see that for sufficiently large $T$, $p(\phi)\geq p(\phi_0)/2 $ for all $\phi\in B_T$ due to the continuity of the prior.

For any $\phi\in B_T$:
\begin{equation}
    l_T(\phi,X_\fullT)-l_T(\phi_0,X_\fullT) \geq -C_LT \frac{1}{\sqrt{T}} = -C_L\sqrt{T},
\end{equation}
which indicates
\begin{equation}
    \inf_{\phi\in B_T} (l_T(\phi, X_\fullT)-l_T(\phi_0, X_\fullT)) \geq -C_L\sqrt{T}.
\end{equation}

As a result, the denominator satisfies:
\begin{align}
    \int_{\Theta} e^{l_T(\phi,X_\fullT)-l_T(\phi_0,X_\fullT)} p(\phi) d\phi &\geq \int_{B_T} e^{l_T(\phi,X_\fullT)-l_T(\phi_0,X_\fullT)} p(\phi) d\phi \\
    &\geq \frac{1}{2}\int_{B_T} e^{-C_L\sqrt{T}} p(\phi_0) d\phi \\
    &=  \frac{e^{-C_L\sqrt{T}} p(\phi_0)}{2} \frac{4\pi}{3T^{3/2}}.
\end{align}

Finally, 
\begin{equation}
    P(A\sube|X_\fullT) \leq \frac{3T^{3/2}}{2\pi p(\phi_0)}  e^{C_L\sqrt{T}-\frac{\kappa(\epsilon)}{2}T}  \asto 0  \quad \text{as } T\to\infty. \quad \blacksquare
\end{equation}

Therefore, \begin{equation}
    \Vert\hat{\phi}_T-\phi_0\Vert\leq \scrE[\Vert\phi-\phi_0\Vert \vert X_\fullT] \leq \epsilon + diam(\Theta) P(A_\epsilon |X_\fullT).
\end{equation}

Let $T\to\infty$, and then let $\epsilon\to 0$, we have $\Vert \hat{\phi}_T-\phi_0\Vert \asto 0$, which proves the estimator consistency $\square$.

\subsection*{Proof for Proposition \ref{pp:endogV-nonidentify}.}
As discussed in Lemma \ref{lm:one-for-all}, we only need to consider a model with $N=1$. Again, we omit the superscript and denote the two value functions as $V_t$ and $V_t'$. 

Since two parameter sets share the same $\rho$, and the MNL model is shift-invaraint, as long as
\begin{equation}
V_t(i)-V_t(1) = V_t'(i)-V_t'(1), \ \forall i=2,3,...,M, t\in[T]
\end{equation}
$(\eta,\theta,\rho,V_1)$ and $(\eta,\theta,\rho,V_1')$ will always generate the same choice probability, hence are not identifiable. We now prove this by induction on $t$. 

First, it automatically holds for $t=1$. Suppose it also holds for $t=k$. Then, the exponential smoothing indicates
\begin{equation}
    V_{k+1}'(i) - V_{k+1}'(1)=(1-\eta) (V_k'(i)-V_k'(1))+\eta (c_k(i) - c_k(1)),
\end{equation}
and
\begin{equation}
    V_{k+1}(i) - V_{k+1}(1)=(1-\eta) (V_k(i)-V_k(1))+\eta (c_k(i) - c_k(1)),
\end{equation}
for all $i=2,3,...,M, t\in[T]$. Therefore, the equation also holds for $t=k+1$. By induction, we prove the proposition. $\square$

\subsection*{Proof for Theorem \ref{thm:endogenize}.}
Similar to our proof for Theorem \ref{thm:basic}, it suffices to consider the case $N=1$, and $\rho$ is identifiable.

Suppose the model is not identifiable. That is, two parameter sets $(\eta, \theta, \rho, \delta ) \neq(\eta',\theta',\rho, \delta ')$ have the same choice probabilities, where $\delta$ and $\delta'$ represent the vector form of the initial value parameters. Applying Equation (\ref{eq:master}) at time $t$ and $t+1$, together with the updating equation, yields
\begin{equation}
\label{eq:endog-help1}
    (\eta'-\eta)\theta\Delta V_t(\cdot,1) + (\theta\eta-\theta'\eta')\Delta c_t(\cdot,1) = 0, \qquad t=1,\ldots,T-1.
\end{equation}

Suppose $\eta\neq\eta'$, and define \begin{equation} 
q= \frac{\theta'\eta'-\theta\eta}{\eta'-\eta}. 
\end{equation} 
Equation~(\ref{eq:endog-help1}) implies 
\begin{equation} 
\label{eq:endog-help2} \theta\Delta V_t(\cdot,1) = q\Delta c_t(\cdot,1), \qquad t=1,\ldots,T-1. 
\end{equation}

Suppose $q=0$. For the day $s$ specified in the first condition of Assumption~\ref{ass:richer}, Equation~(\ref{eq:endog-help2}) gives \begin{equation} \theta\Delta V_s(\cdot,1) = \theta\Delta V_{s+1}(\cdot,1) = 0, \end{equation} where $s+1\leq T-1$. The updating equation then implies 
\begin{align} 
\theta\Delta V_{s+1}(\cdot,1)= (1-\eta)\theta\Delta V_s(\cdot,1) + \theta\eta\Delta c_s(\cdot,1)= \theta\eta\Delta c_s(\cdot,1)=0, \end{align} which contradicts $\theta>0$, $\eta>0$, and $\Delta c_s(\cdot,1)\neq0$. 

Therefore, $q\neq0$. Substituting Equation~(\ref{eq:endog-help2}) into the updating equation gives \begin{equation} 
q\Delta c_{t+1}(\cdot,1) = \left[(1-\eta)q+\theta\eta\right] \Delta c_t(\cdot,1), \qquad t=1,\ldots,T-2. 
\end{equation}
Since $q\neq0$, it follows that \begin{equation} 
\Delta c_{t+1}(\cdot,1) = \left( 1-\eta+\frac{\theta\eta}{q} \right) \Delta c_t(\cdot,1), \qquad t=1,\ldots,T-2. 
\end{equation} 
Thus, the scalar $r=1-\eta+\theta\eta/q$ satisfies $\Delta c_{t+1}(\cdot,1)=r\Delta c_t(\cdot,1)$ for every $t=1,\ldots,T-2$, contradicting the second condition in Assumption~\ref{ass:richer}. 

Therefore, $\eta=\eta'$, then
\begin{equation} \eta(\theta-\theta')\Delta c_t(\cdot,1)=0, \qquad t=1,\ldots,T-1. \end{equation}
By the first condition in Assumption~\ref{ass:richer}, there exists $s\in\{1,\ldots,T-2\}$ such that $\Delta c_s(\cdot,1)\neq0$. Since $\eta>0$, it follows that $\theta=\theta'$. Applying Equation~(\ref{eq:master}) at $t=1$ then gives $\theta\delta=\theta'\delta'$ and hence $\delta=\delta'$, which contradicts the starting assumption. Therefore, the model is identifiable. $\square$

\subsection*{Proof for Proposition \ref{pp:partial-pool-identify}.}
We omit the proof as it can be proved following a similar way as in Proposition \ref{pp:partial-pool-endogV-identify}.

\subsection*{Proof for Proposition \ref{pp:partial-pool-endogV-identify}.}
In the pooled model, all commuters share the choice probability. Denote the choice probability generated by parameter $\eta,\theta,\rho,\delta$ as $p_t(i|\eta,\theta,\rho,\delta,\scrF_\tmo)$.

We first present the following lemma. As the observation likelihood is given by a multinomial distribution, this lemma can be proved by Corollary 6.16 in \cite{lehmann1998theory}.
\begin{lemma}
\label{lm:lehmann}
If $P(O_t=o_t|\eta, \theta,\rho, \delta,\scrF_\tmo) = P(O_t=o_t|\eta', \theta',\rho', \delta',\scrF_\tmo)$ for all observation $o_t$ and $t\in[T]$, there must be $p_t(i|\eta, \theta,\rho, \delta,\scrF_\tmo)=p_t(i|\eta', \theta',\rho', \delta',\scrF_\tmo)$ for all $i=0,...,M$ and $t\in[T]$.
\end{lemma}

Now, suppose the model is not identifiable, meaning that there exist two different parameter sets $(\eta, \theta,\rho, \delta)\neq (\eta', \theta',\rho', \delta')$, such that $P(O_t=o_t|\eta, \theta,\rho, \delta,\scrF_\tmo) = P(O_t=o_t|\eta', \theta',\rho', \delta',\scrF_\tmo)$ for all observation $o_t$ and $t\in[T]$. Lemma \ref{lm:lehmann} ensures that these parameter sets also ensure $p_t(i|\eta, \theta,\rho, \delta,\scrF_\tmo)=p_t(i|\eta', \theta',\rho', \delta',\scrF_\tmo)$ for all $i=0,...,M$ and $t\in[T]$, which contradicts the identifiability in Theorem \ref{thm:endogenize}. Therefore, the assumption does not hold, and the model is identifiable. $\square$

\subsection*{Proof for Theorem \ref{thm:equal}.}
The posterior of the anonymized observation model satisfies
\begin{equation}
    p(\eta,\theta,\rho| o_\fullT, c_\fullT) \propto  P(O_\fullT=o_\fullT|\eta,\theta,\rho, c_\fullT) p(\eta)p(\theta)p(\rho),
\end{equation}
where the likelihood function equals: 
\begin{equation}
\begin{split}
    P(O_\fullT=o_\fullT|\eta,\theta,\rho, c_\fullT) &=  \prod_{t=1}^T P(O_t=o_t|\eta,\theta,\rho, \scrF_\tmo), \\
    &= \prod_{t=1}^T \frac{N!}{o_t(0)!\cdots o_t(M)!} \prod_{i=0}^M p_t(i)^{o_t(i)}.
\end{split}
\end{equation}

Thus, the log-posterior satisfies:
\begin{equation}
    \log p(\eta,\theta,\rho| o_\fullT, c_\fullT)=\sum_{t=1}^T\sum_{i=0}^M o_t(i)\log p_t(i) +\log p(\eta,\theta,\rho) \underbrace{-\sum_{t=1}^T\sum_{i=0}^M \log o_t(i)! +T\log N}_{\text{Given } o_t, \text{ it is a constant}}+\text{const}
\end{equation}
    
Meanwhile, for the fully observable model, we denote the choice probability as $p_t^n(i)=p_t(i)$ for all $n\in[N]$. We have
\begin{align}
    p(\eta,\theta,\rho|x_\fullT, c_\fullT) &\propto P(X_\fullT=x_\fullT|\eta,\theta,\rho, c_\fullT) p(\eta)p(\theta)p(\rho),
\end{align}
and further
\begin{equation}
    P(X_\fullT=x_\fullT|\eta,\theta,\rho, c_\fullT) = \prod_{t=1}^T \prod_{n=1}^N P(X^n_t=x^n_t|\eta,\theta,\rho, \scrF_\tmo)=  \prod_{t=1}^T \prod_{n=1}^N p_t(x_t^n)=  \prod_{t=1}^T \prod_{i=0}^M p_t(i)^{o_t(i)},
\end{equation}
where the final equality holds because in $\prod_{n=1}^N p_t(x_t^n)$, there are $o_t(i)$ number of $p_t(i)$, for $i=0,...,M$. 

Therefore, the log-posterior satisfies:
\begin{equation}
    \log p(\eta,\theta,\rho| x_\fullT, c_\fullT)=\sum_{t=1}^T\sum_{i=0}^M o_t(i)\log p_t(i) +\log p(\eta,\theta,\rho) + \text{const},
\end{equation}
which only has a constant shift relative to the anonymized case. As the posteriors need to sum up to one, they must be identical. 
$\square$

\subsection*{Proof for Proposition \ref{pp:approximate}.}
For expectations:
\begin{equation}
    \scrE[\otild_t] = N \Bar{p}_t = \sum_{n=1}^N p_t^n=\scrE[o_t], \ \forall t\in[T].
\end{equation}

For variance:
\begin{equation}
    \text{Cov}(\otild_t) = N\left(\text{diag}(\Bar{p}_t) - \Bar{p}_t \Bar{p}_t^T \right)
\end{equation}

We have
\begin{align}
    \text{Cov}(\otild_t) - \text{Cov}(o_t) &= N\left(\text{diag}(\Bar{p}_t) - \Bar{p}_t \Bar{p}_t^T \right) - \sum_{n=1}^N \left( \text{diag}(p_t^n) - p_t^n (p_t^n)^T \right) \\
    &= \sum_{n=1}^N p_t^n (p_t^n)^T - N\Bar{p}_t \Bar{p}_t^T,
\end{align}
where the diagonal terms cancel out because $\text{diag}$ is a linear operator. 

Let $\delta_t^n=p_t^n-\Bar{p}_t$. Then, $\sum_{n=1}^N \delta_t^n=0$. Expand:
\begin{align}
    \sum_{n=1}^N p_t^n (p_t^n)^T &= \sum_{n=1}^N (\Bar{p}_t+\delta_t^n) (\Bar{p}_t+\delta_t^n)^T = N \Bar{p}_t  \Bar{p}_t ^T + \sum_{n=1}^N \delta _t^n (\delta _t^n)^T.
\end{align}

As each matrix $\delta _t^n (\delta _t^n)^T$ is positive semi-definite (PSD), the sum of PSD matrices is also PSD, which proves that $\text{Cov}(\otild_t) \succeq \text{Cov}(o_t)$. $\square$

\section*{Appendix B. Additional Results}
\subsection{Bayesian Diagnostics}
\label{app:ss:diagnostics}
Here we present the sampling diagnostics for the pooled model with $N=3,T=30$. 

The first metric is the split $\hat{R}$, which assesses convergence by comparing within-chain and between-chain variance. Figure \ref{fig:rhat_histograms} shows that all the empirical distribution concentrates at 1.00, indicating satisfactory convergence. 

\begin{figure}[!ht]
    \centering
    \includegraphics[width=0.6\linewidth]{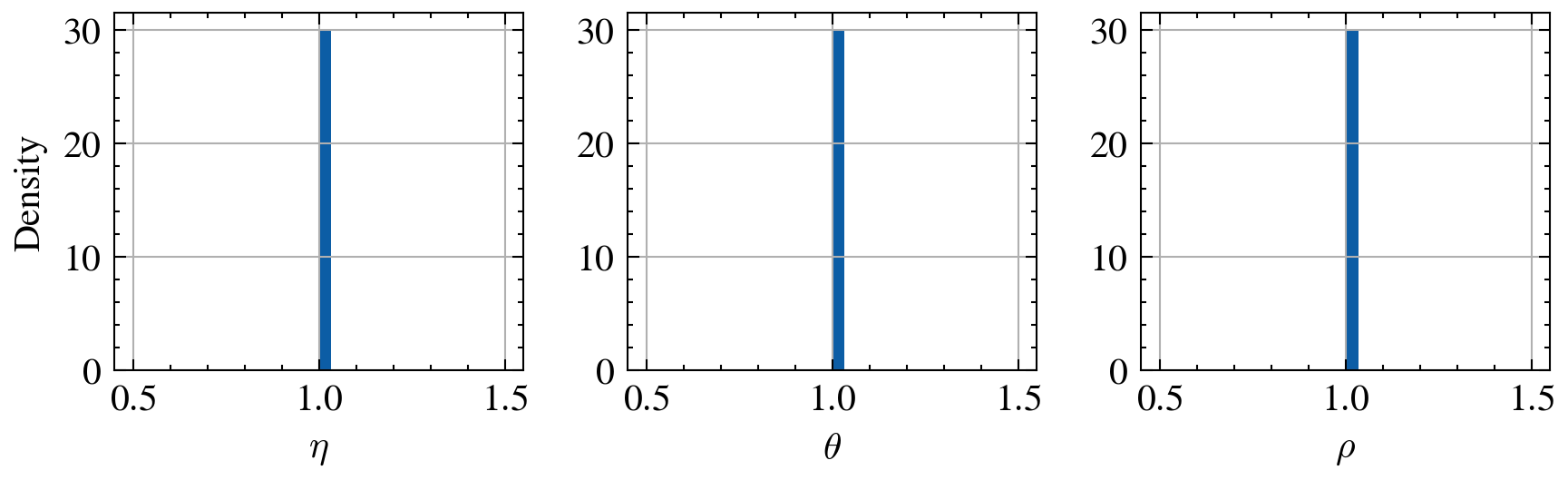}
    \caption{Histogram of $\hat{R}$.}
    \label{fig:rhat_histograms}
\end{figure}

Figure \ref{fig:ess_histograms} presents the distribution of effective sample size (ESS), which measures sampling efficiency by accounting for autocorrelation among posterior draws. For all three parameters, ESS values are predominantly above 2,500, indicating efficient exploration of the posterior distribution.

\begin{figure}[!ht]
    \centering
    \includegraphics[width=0.6\linewidth]{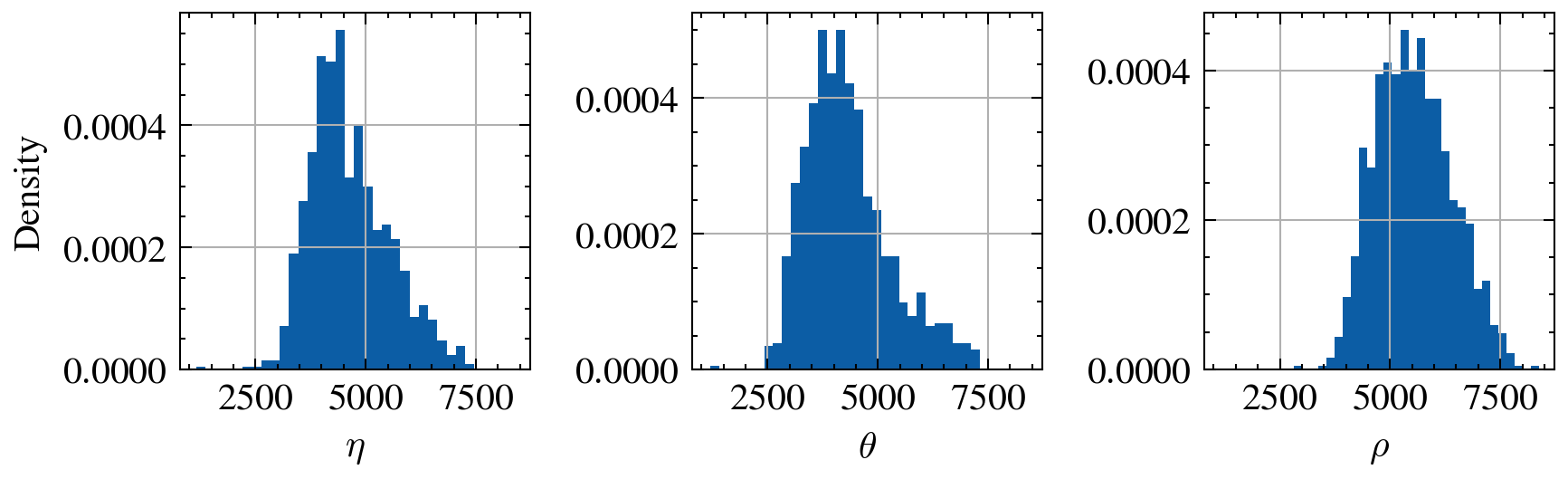}
    \caption{Histogram of ESS.}
    \label{fig:ess_histograms}
\end{figure}

Figure \ref{fig:rank_histograms} presents the empirical distribution of the normalized rank of the true parameter value within the posterior draws across simulation replications. The rank histograms are approximately uniform for all parameters, suggesting that the posterior distributions are well calibrated across repeated simulations and exhibit no systematic tendency to place the true values disproportionately near either tail.

\begin{figure}[!ht]
    \centering
    \includegraphics[width=0.6\linewidth]{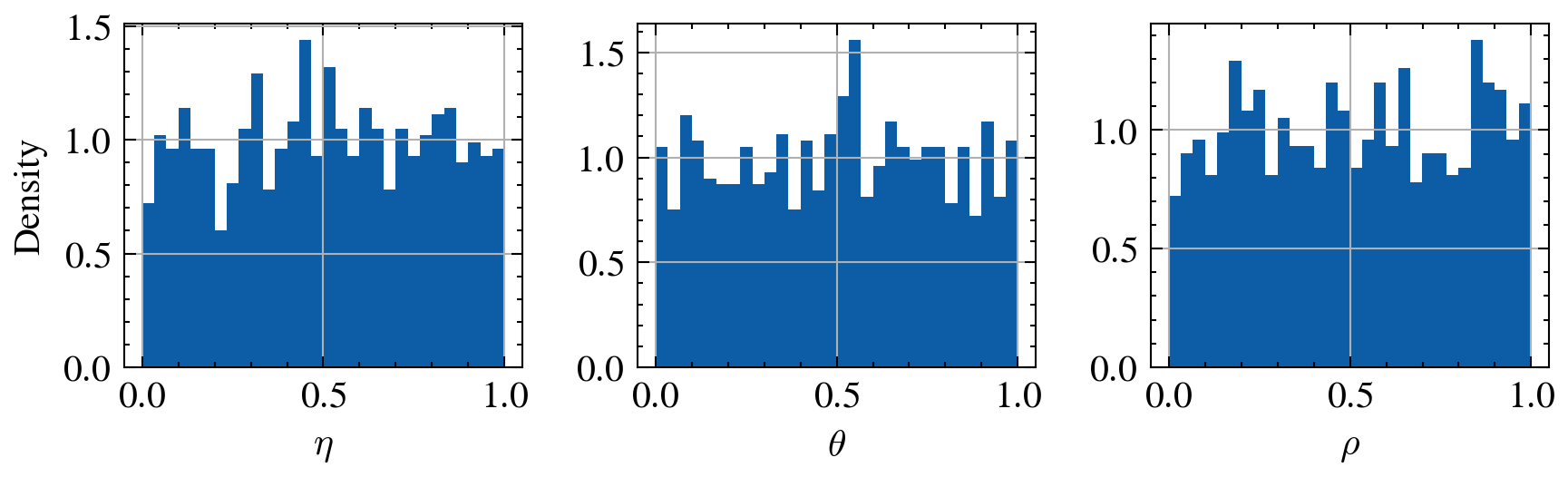}
    \caption{Histogram of rank.}
    \label{fig:rank_histograms}
\end{figure}

\subsection{Robustness to Misspecification}\label{app:ss:robust}
\subsubsection{Behavioral model misspecification}
Besides the correctly specified Horowitz model, we consider an alternative data-generating process based on the Smith model \citep{smith1984stability}: starting from a uniform random choice over paths, on day $t$, a traveler who previously chose path $i$ switches to another path $j\neq i$  with probability:
\begin{equation}
P(X_t^n=j|X_\tmo^n=i, \tau, \epsilon, \scrF_\tmo) 
=\begin{cases}
        \tau (c_\tmo(i)-c_\tmo(j)), \ & \text{if }  c_\tmo(j)\leq c_\tmo(i),\\
        \epsilon, \ & \text{if } c_\tmo(j)>c_\tmo(i),
    \end{cases}
\end{equation}
and thus
\begin{equation}
    p(X_t^n=i|X_\tmo^n=i, \tau, \epsilon, \scrF_\tmo)= 1-\sum_{j\neq i, j\in[M]} P(X_t^n=j|X_\tmo^n=i, \tau, \epsilon, \scrF_\tmo).  
\end{equation}
Intuitively, commuters on a high-cost option will switch to lower-cost options at a rate that is proportional to the cost difference. 

For clarity, we assume no demand variation (i.e., $\rho=0$) and estimate the other two more challenging parameters. For each data-generating process, we consider three levels of behavioral responsiveness. The Smith parameters $(\tau,\epsilon)$ are set to $(0.05,0.10)$, $(0.10,0.05)$, and $(0.12,0.01)$, while the corresponding Horowitz parameters $(\eta,\theta)$ are set to $(0.10,0.5)$, $(0.20,1.0)$, and $(0.40,2.0)$. We additionally vary the population size between $N=10$ and $N=20$ and the training horizon between 30 and 80 days. Each fitted model is evaluated over the immediately following $T_{test}=20$ held-out days, and each configuration is repeated three times. Predictive performance is measured using the root mean squared error (RMSE) between the choice probabilities predicted by the estimated Horowitz model and the exact probability trajectories generated by the corresponding data-generating process, averaged across the three routes and 20 held-out days.

Figure \ref{fig:robust_posterior_smith} shows posterior samples under two representative settings. Relative to the correctly specified case in Figure \ref{fig:posterior}, the posterior distributions under behavioral misspecification are more irregular, reflecting the mismatch between the assumed model and the data-generating process. Nevertheless, increasing the amount of information through larger $N$ and longer $T_{train}$ yields more concentrated posterior distributions.

\begin{figure}[!ht]
    \centering
    \includegraphics[width=0.5\linewidth]{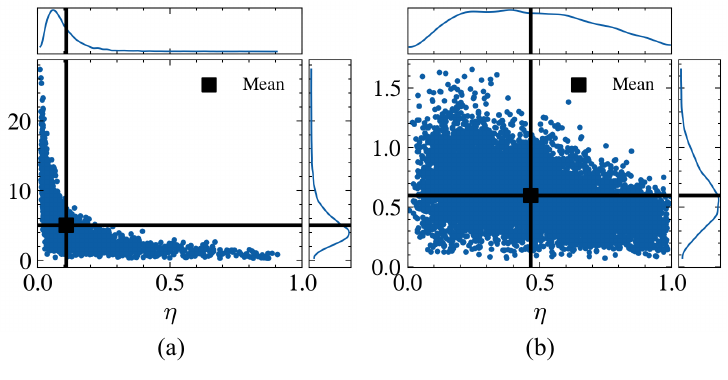}
    \caption{Sampled posteriors. (a): $N=10, T_{train}=30$; (b): $N=20, T_{train}=80$.}
    \label{fig:robust_posterior_smith}
\end{figure}

\subsubsection{Shift in parameter generation}
In practice, the modeler rarely knows the true distribution from which behavioral parameters are generated. To examine sensitivity to prior misspecification, we generate data under two alternative parameter distributions that differ from the estimation priors:
\begin{itemize}
\item Same distribution families but different parameters:

$\log \left(\frac{\eta^{(s)}}{1-\eta^{(s)}}\right) \sim N(-0.85, 0.5), \log \theta^{(s)}\sim N(1.1, 0.6), \log \left(\frac{\rho^{(s)}}{1-\rho^{(s)}}\right) \sim N(-1.5, 0.7);$
\item Different distribution families:

$
    \eta^{(s)}\sim Beta(2, 5), \theta^{(s)}\sim Gamma(2,1), \rho^{(s)}\sim Beta(2,8).
$
\end{itemize}

Figure \ref{fig:distribution_changes} visualizes these distributions. As expected, the shapes differ substantially from the estimation prior.
\begin{figure}[!ht]
    \centering
    \includegraphics[width=0.6\linewidth]{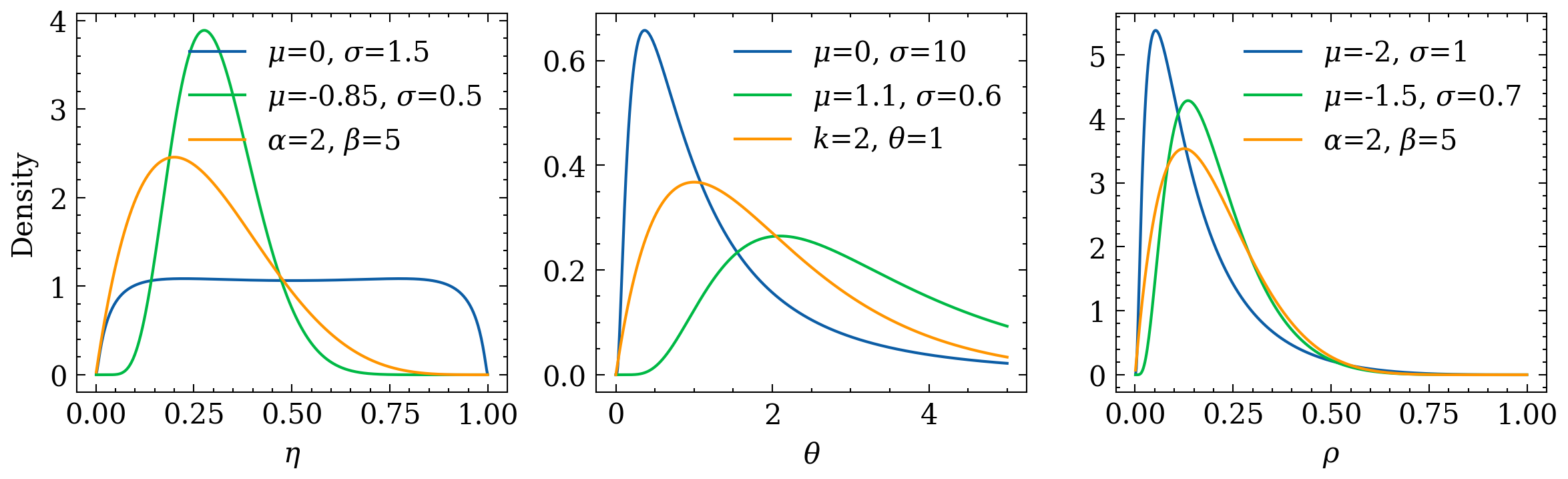}
    \caption{Different distributions: prior (green), same family but different parameters (green), different distribution family (orange)}
    \label{fig:distribution_changes}
\end{figure}

Figure \ref{fig:diffScenarios} summarizes estimation performance under these scenarios. Overall, misspecification degrades performance, increasing bias and reducing coverage, but the deterioration remains within an acceptable range. In contrast, the effect on interval width is comparatively small: interval width mainly reflects how much information can be extracted from a given dataset under the chosen model, and moderate prior misspecification has limited influence on that extraction efficiency. Among the parameters, $\rho$ is the least affected by the distribution shift, consistent with its relatively clean identifiability through the travel/no-travel mechanism.
\begin{figure}[!ht]
    \centering
    \includegraphics[width=1\linewidth]{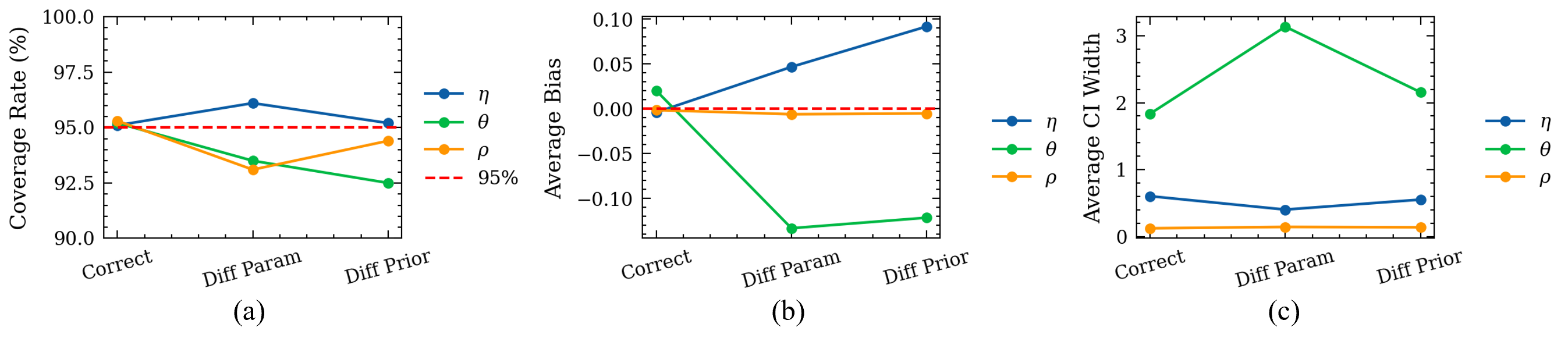}
    \caption{Estimation performance under different scenarios}
    \label{fig:diffScenarios}
\end{figure}

\subsubsection{Behavior heterogeneity}
We next examine the robustness of the pooled estimator when the data are generated by travelers with heterogeneous parameters. Again, we assume no demand variation. We set $N=10$, $T_{train}=30$, and predict choice probabilities over the next $T_{test}=20$ days. Each traveler $n$ independently draws individual parameters according to:
\begin{equation}
    \log\left( {\eta^n}/{(1-\eta^n)} \right) \sim \scrN(0, 1.5^2), \ \log(\theta^n) \sim \scrN(0,1^2)
\end{equation}

Figure \ref{fig:robust_posterior} shows the resulting posterior samples together with the individual-level true parameter values. The pooled posterior lies near the center of the individual parameter cloud, indicating that the pooled model recovers a central aggregate representation of heterogeneous behavior.
\begin{figure}[!ht]
    \centering
    \includegraphics[width=0.25\linewidth]{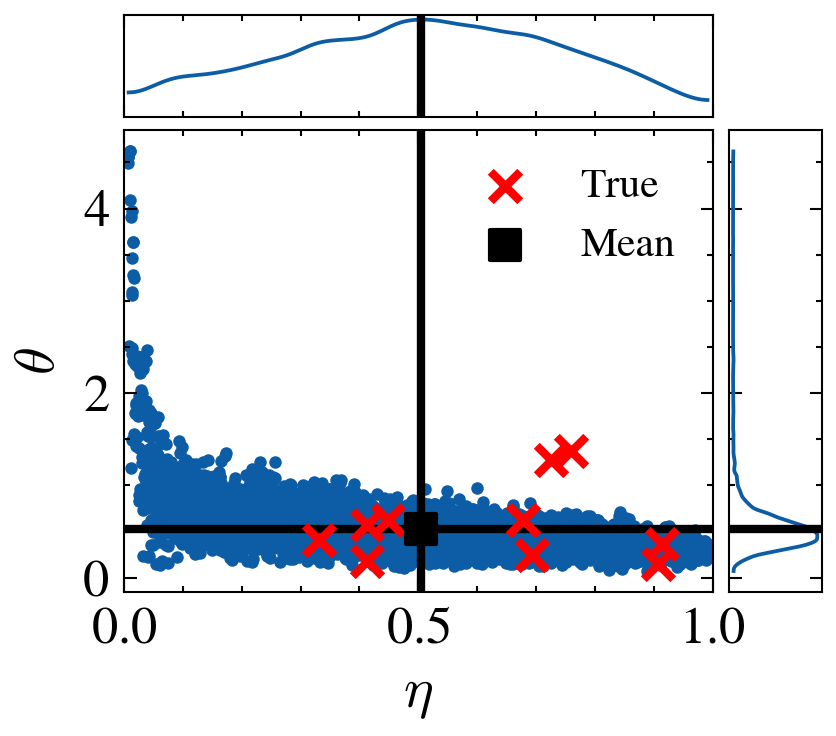}
    \caption{Sampled posteriors.}
    \label{fig:robust_posterior}
\end{figure}

Figure \ref{fig:robust_predict} reports extrapolation performance. Prediction is highly accurate, indicating that even when individual-level behavioral interpretation is limited by pooling, aggregate outcomes can still be predicted well from a pooled approximation. 
\begin{figure}[!ht]
    \centering
    \includegraphics[width=0.8\linewidth]{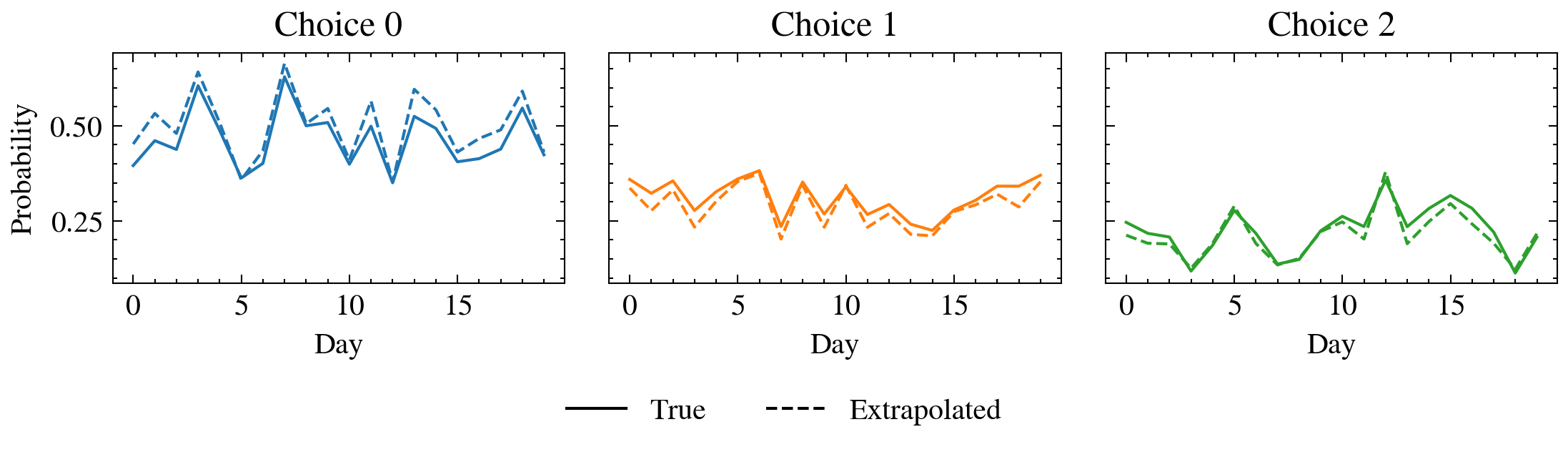}
    \caption{Extrapolated choice probabilities in $T_{test}=20$ days.}
    \label{fig:robust_predict}
\end{figure}


\subsection{Hierarchical Model Estimations}\label{app:ss:hier}
Figure \ref{fig:hier_coverage} reports empirical coverage for 95\% intervals and shows rates close to nominal. Some degradation appears for large $N$ or $T$, likely due to increased computational difficulty (e.g., divergences or poorer mixing), which can reduce effective sample size and slightly distort interval calibration.

\begin{figure}[!ht]
    \centering
    \includegraphics[width=0.85\linewidth]{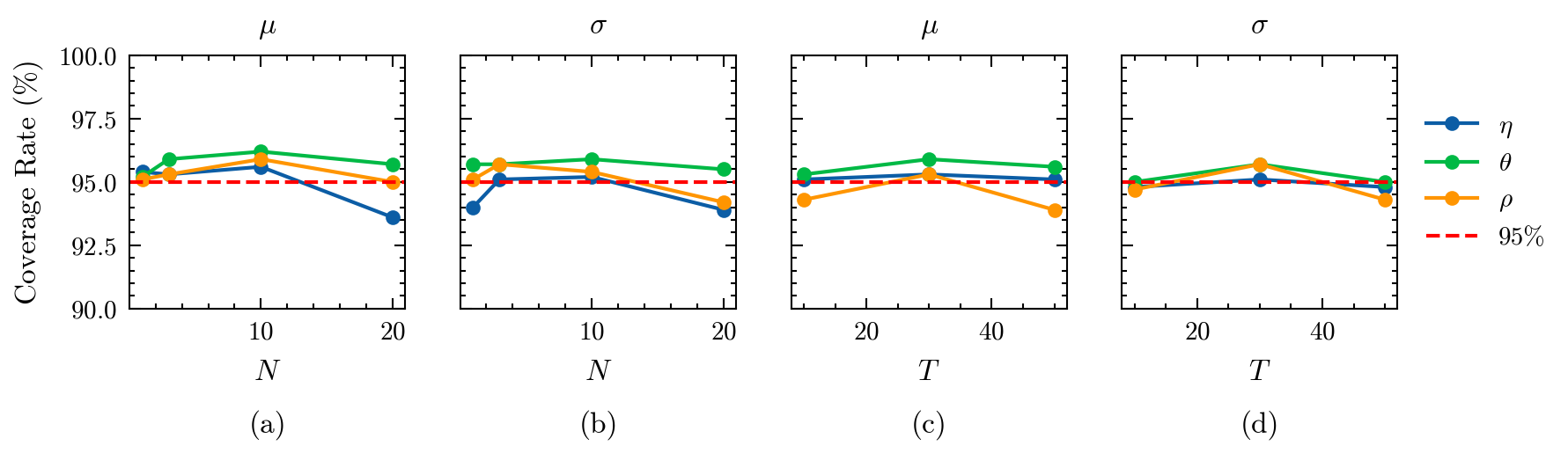}
    \caption{Coverage rate. (a) and (b): Fix $N=3$ and vary $T$; (c) and (d): Fix $T=30$ and vary $N$. }
    \label{fig:hier_coverage}
\end{figure}

Figures \ref{fig:hier_ci_width_N} and \ref{fig:hier_ci_width_T} report average widths of the 95\% CIs. Width decreases monotonically with $N$, reinforcing that population size is the main driver for learning heterogeneity. When $T$ increases from 30 to 50, widths for most hyperparameters stop decreasing, suggesting that roughly 30 days already contain sufficient temporal information in this synthetic setting and additional days yield diminishing returns.

\begin{figure}[!ht]
    \centering
    \includegraphics[width=1\linewidth]{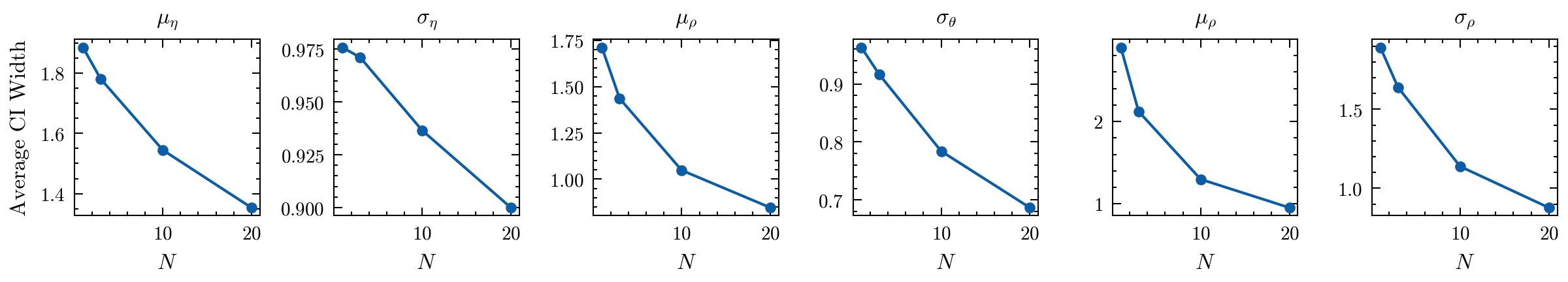}
    \caption{Average CI width of hyperparameters against varying $N$.}
    \label{fig:hier_ci_width_N}
\end{figure}

\begin{figure}[!ht]
    \centering
    \includegraphics[width=1\linewidth]{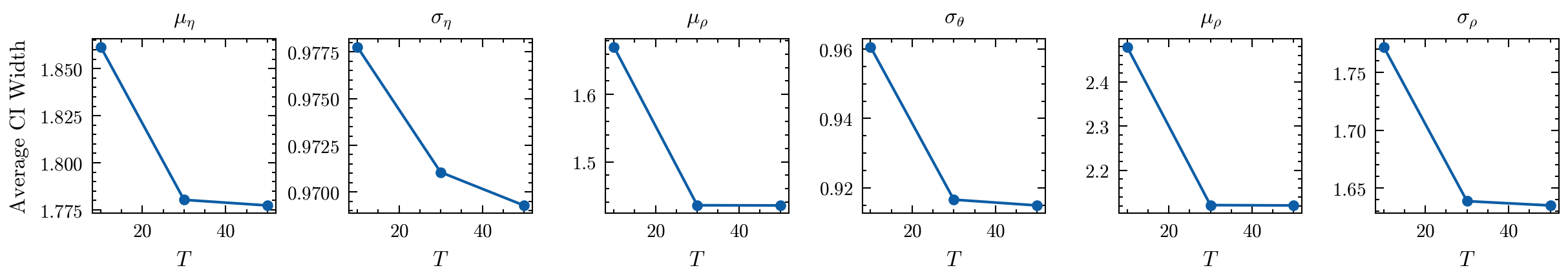}
    \caption{Average CI width of hyperparameters against varying $T$.}
    \label{fig:hier_ci_width_T}
\end{figure}

\subsection{Controlled Lab Experiments}\label{app:ss:lab}
Figure \ref{fig:wij_noinfo} presents the predictive performance of the model with fixed initial values on the no-information data. 

Figure \ref{fig:rep_freq_comparison} presents the posterior predictive performance for the four participant types. In general, the predictive bands cover the realized frequencies well. In contrast, Figure \ref{fig:fit_40day} shows the results for the 40-day coarse path-level data, where the predictive mean is almost stationary and fails to capture path-count variations.

\begin{figure}[!ht]
    \centering
    \includegraphics[width=0.7\linewidth]{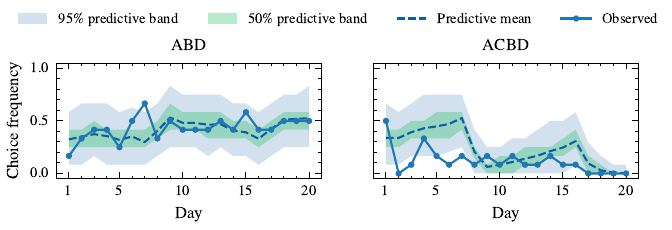}
    \caption{Posterior predictive performance for no information and fixed initial values}
    \label{fig:wij_noinfo}
\end{figure}

\begin{figure}[!ht]
    \centering
    \includegraphics[width=0.8\linewidth]{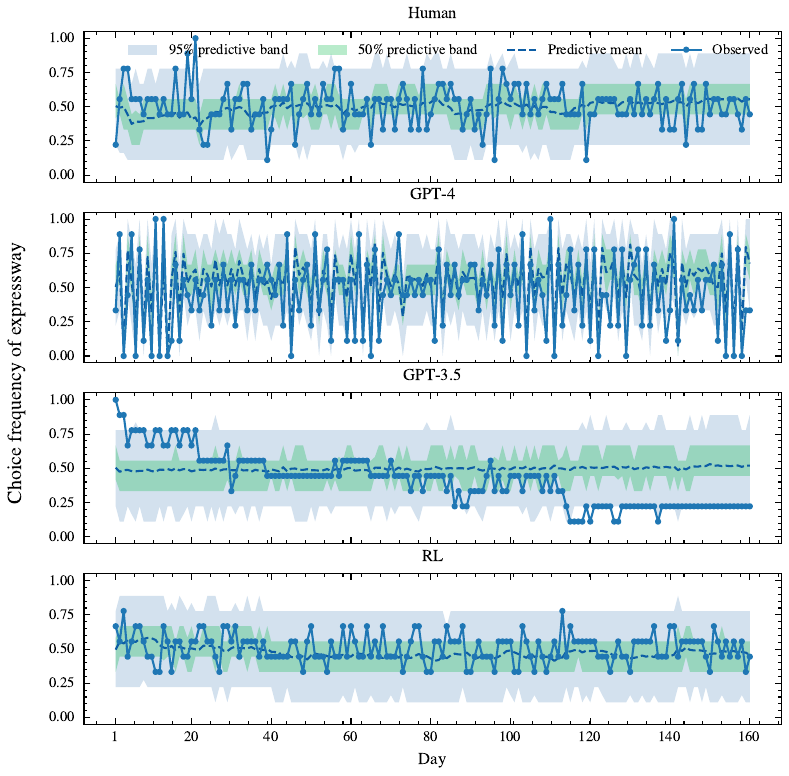}
    \caption{Posterior predictive performance for different participant types}
    \label{fig:rep_freq_comparison}
\end{figure}

\begin{figure}[!ht]
    \centering    \includegraphics[width=0.7\linewidth]{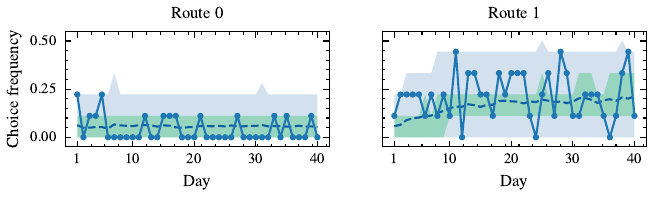}
    \caption{Posterior predictive performance for path-level human data}
    \label{fig:fit_40day}
\end{figure}

\subsection{Trajectory Data}\label{app:ss:traj}
Figure \ref{fig:two_od} gives the other two OD pairs in the experiment. Both involve two routes: one primarily uses highways and the other utilizes local roads.
\begin{figure}[!ht]
    \centering
    \includegraphics[width=0.55\linewidth]{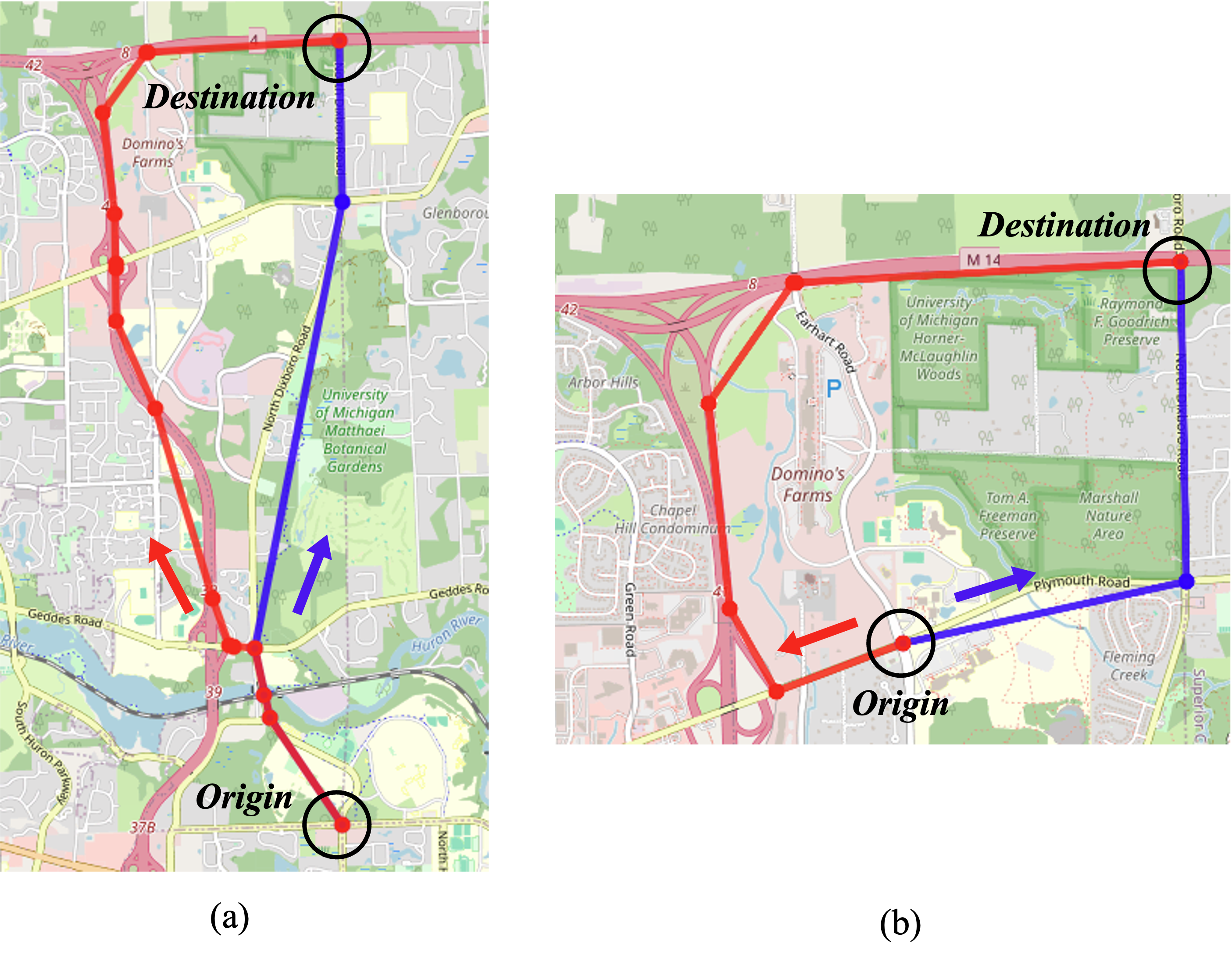}
    \caption{The other two OD pairs}
    \label{fig:two_od}
\end{figure}

\end{document}